\documentclass[12pt,a4]{article}
\usepackage[utf8]{inputenc}
\usepackage{multirow}
\usepackage{eqnarray}
\usepackage{amsmath,graphicx,subfigure}
\usepackage{amssymb,amsfonts,anysize,url}
\newenvironment{proof}{\paragraph{Proof:}}{\hfill$\square$}
\usepackage{caption}
\usepackage{mathrsfs}
\usepackage{graphicx}
\usepackage[english]{babel}

\usepackage[affil-sl]{authblk}

\usepackage{xcolor}
\usepackage[colorlinks=true,linkcolor=blue]{hyperref}%
\newtheorem{definition}{Definition}[section]
\newtheorem{theorem}[definition]{Theorem}

\newtheorem{corollary}[definition]{Corollary}
\newtheorem{remark}[definition]{Remark}
\newtheorem{proposition}[definition]{Proposition}

\newtheorem{lemma}[definition]{Lemma}

\numberwithin{equation}{section}

\allowdisplaybreaks

\title{Nonlocal critical exponent singular {problems under} mixed Dirichlet-Neumann boundary {conditions}}
\date{}
\begin{document}

\author[1]{Tuhina Mukherjee}
 \author[2]{Patrizia Pucci}
 \author[1]{Lovelesh Sharma}
 \affil[1]{\small Department of Mathematics,
Indian Institute of Technology Jodhpur}
\affil[2]{\small Dipartimento di Matematica e Informatica,
Universit\`a degli Studi di Perugia }
\providecommand{\keywords}[1]
{
  \small	
  \textbf{\textit{Keywords---}} #1
}
\maketitle
\begin{abstract}
\vspace{1cm}
 In this paper, we study the following singular problem,
 under mixed Dirichlet-Neumann boundary conditions, and  involving
the fractional Laplacian
 \begin{equation*} \label{1}
      \begin{cases}
      (-\Delta)^{s}u =  \lambda u^{-q} + u^{2^*_s-1}, \quad u>0 \quad
 \text{in }\Omega,\\
 \mathcal A(u) = 0 \quad \text{on}~ \partial\Omega =  \sum_{D} \cup \sum_{\mathcal{N}},
  \end{cases} \tag{$P_\lambda$}
 \end{equation*}
 where  $\Omega \subset \mathbb{R}^N$  is a bounded domain with smooth boundary $\partial{\Omega}$, $1/2<s<1$, $\lambda >0$ is a  real parameter, $ 0 < q < 1 $, $N>2s$, $2^*_s=2N/(N-2s)$  and
 $$\mathcal{A}(u)= u \mathcal{X}_{\sum_{D}} + {\partial_{\nu}u}\mathcal{X}_{ \sum_{\mathcal{N}}}, \quad{\partial_{\nu}=\frac{\partial }{\partial{\nu}}}.$$
 Here $\sum_{D}$, $\sum_{\mathcal{N}}$ are smooth $(N-1)$ dimensional
 submanifolds of $\partial \Omega$ such that $\sum_{D} \cup \sum_{\mathcal{N}}= \partial\Omega$, $\sum_{D} \cap \sum_{\mathcal{N}}=  \emptyset $ and $\sum_{D} \cap \overline{\sum_{\mathcal{N}}} = \tau'$ is a smooth $(N-2)$ dimensional submanifold of $\partial{\Omega}$. Within a suitable range of $\lambda$, we establish existence of at least two opposite energy solutions for \eqref{1} using the standard Nehari manifold technique.
\end{abstract}
\maketitle
\numberwithin{equation}{section}
\allowdisplaybreaks

\section{Introduction}
We investigate the existence and multiplicity of solutions to the following problem
\begin{equation*}\label{Main problem}
\begin{cases}
          (-\Delta)^{s}u =  \lambda u^{-q} + u^{2^*_s-1}, \quad u>0 \quad
 \text{in }\Omega,\\
 \mathcal{A}(u) = 0 \quad \text{on}~ \partial\Omega =  \sum_{D} \cup \sum_{\mathcal{N}},\end{cases}\tag{$P_\lambda$}
\end{equation*}
where $\lambda > 0$ is a real parameter, $2^*_s=\frac{2N}{N-2s}$, $N>2s$, $s\in (1/2,1)$, $q\in (0,1)$, $\Omega \subset \mathbb{R}^N$  is a bounded domain with smooth boundary $\partial{\Omega}$ and
the mixed Dirichlet-Neumann boundary conditions are given by
$$\mathcal{A}(u)= u \mathcal{X}_{\sum_{D}}
+ {\partial_\nu u}\mathcal{X}_{ \sum_{\mathcal{N}}},
\quad {\partial_{\nu}=\frac{\partial }{\partial{\nu}}}.$$
Here, $\nu$ is the outward unit normal vector to $\partial\Omega$ and $\mathcal{X}_{M}$ is  the indicator function of the set $M$. The boundary manifolds $\sum_{D}$ and {$\sum_{\mathcal{N}}$} are assumed to be such that
$\sum_{D}$, $\sum_{\mathcal{N}}$ are smooth $(N-1)$ dimensional submanifolds of $\partial \Omega$ and $\sum_{D}$ is a closed manifold of positive $(N-1)$ dimensional Lebesgue measure, where  $|\sum_{D}| = \gamma \in(0,| \partial{\Omega}|)$ and $\sum_{D} \cup \sum_{\mathcal{N}}= \partial \Omega$ , $\sum_{D} \cap \sum_{\mathcal{N}}=  \emptyset $ and  $\sum_{D} \cap \overline{\sum_{\mathcal{N}}} = \tau'$ is a smooth  $(N-2)$ dimensional
submanifold of $\partial{\Omega}.$ 
 The study of fractional Laplacian is a famous topic in mathematical analysis and probability. The application of this theory is the obstacle problem, which arises in a variety of fields such as population dynamics, chemical reactions in liquids, and image processing. Nonlinear problems driven by fractional Laplace operators form a common feature in many mathematical models  of biology and finance. Many articles on the study of obstacle problems focus on the fractional power of the Laplace operator which appears as the infinitesimal generator of the L\'{e}vy stable diffusion process, for instance, refer \cite{MR2064019,MR1918950,MR2178227,MR2480109} for more details.

Over the past few decades, mixed Dirichlet and Neumann boundary problems associated with elliptic equations
  \begin{equation}\label{Main equation} { -\Delta u = g(x,u)}  ~~\;\text{in}\;~~ \Omega,
 \end{equation}
 has been extensively studied. In particular, the study of wave phenomena, heat transfer, electrostatics, elasticity and hydrodynamics in physics and engineering often involves elliptic problems with mixed Dirichlet-Neumann boundary conditions. For example, one can see \cite{MR647124,MR1472351,MR2378088} and the references therein for motivation. For a detailed study on the Laplace operator with mixed Dirichlet-Neumann boundary conditions, we refer \cite{MR2100910, MR2238024,MR3912863, MR1837538, MR1971262}. Motivated by this, researchers have recently started looking at the non-local version of elliptic equations with mixed boundary conditions but there is only limited literature available so far, such as \cite{MR3912863, MR4065090, MR2078192, MR4319042, MR4343111}. 
 The non local fractional Laplace operator, for $s\in (0,1)$, has been defined through spectral decomposition, that is 
\[(-\Delta)^{s} v =\sum_{j\geq{1}} \lambda^s_{j}\langle v,\phi_{j}\rangle \phi_{j},\]
where $(\lambda_j,\phi_j)$ are eigenvalue-eigen function pair of $(-\Delta)$ with necessary boundary conditions (see Section $2$) as well as through singular integrals, that is 
\[{(- \Delta)^{s}u(x)} = C(N,s)~ P.V. \int_{\mathbb{R}^N} {\dfrac{u(x)-u(y)}{|x-y|^{N+2s}}} ~ dy\]
in literature. Although stated with same name, these operators are in general different, we refer \cite{MR3233760} for details. In \cite{MR3912863}, authors reconsidered the celebrated  Brezis-Nirenberg problem  utilizing the spectral fractional Laplace operator under mixed Dirichlet-Neumann boundary conditions, precisely
\begin{equation}\label{BN}
\begin{cases}
          (-\Delta)^{s}u =  \lambda u + u^{2^*_s-1}, \quad u>0 \quad
 \text{in }\Omega,\\
 \mathcal{A}(u) = 0 \quad \text{on}~ \partial\Omega =  \sum_{D} \cup \sum_{\mathcal{N}},\end{cases}
\end{equation}
where all parameters and notations means same as in \eqref{1}. They established parallel results as in Brezis-Nirenberg original article with a striking difference that \eqref{BN} has a solution when $\lambda =0$ contrasting the Dirichlet case.
Authors in \cite{MR4319042} proves that any solution of
\begin{equation}\label{sm1}
\begin{cases}
          (-\Delta)^{s}u =  f\quad
 \text{in }\Omega,\\
 \mathcal{A}(u) = 0 \quad \text{on}~ \partial\Omega =  \sum_{D} \cup \sum_{\mathcal{N}},\end{cases}
\end{equation}
is in $C^\gamma(\bar{\Omega})$ for $\gamma \in (0,\frac12)$, when $f\in L^p(\Omega)$, $p>\frac{N}{2s}$, via De Giorgi's technique.
The essential strong maximum principles for problems of the type \eqref{sm1} have been established in
 \cite{MR4065090} and \cite{MR4343111} for fractional Laplacian defined via singular integrals and spectral decomposition, respectively.

 Now, we head towards importance of elliptic problems consisting singular nonlinearity or blow-up nonlinearity and critical exponent term, especially of the type
 \begin{equation}\label{singular}
\begin{cases}
          -\Delta u =  \lambda u^{-q}+ u^r,\quad u>0\quad
 \text{in }\Omega,\\
u = 0 \quad \text{on}~ \partial\Omega,\end{cases}
\end{equation}
 where $q>0$ and $r\in (1,\frac{N+2}{N-2}]$. The article by Crandall, Rabinowitz and Tartar was the starting point of such singular problems, \cite{MR427826}.  Due to presence of singular term, the associated energy functional is no longer $C^1$, therefore the standard critical point theory fails which makes such problems very interesting. For $q\in (0,1)$, Haitao et al. in \cite{MR1964476} studied \eqref{singular} with the sub-super solution techniques whereas authors in \cite{MR2099611} studied the same problem \eqref{singular} via Nehari manifold technique. Both articles establish existence of at least two solutions and regularity results. The article \cite{MR2446183} deals with the delicate $q\geq1$ case, where again existence of at least two weak solutions were established using the advanced theory of nonsmooth analysis. The nonlocal extension of these problems, that is \eqref{singular} in presence of fractional Laplacian {defined via singular integrals}, has been studied in \cite{MR4018284, MR3466525, MR3356049,MR3680366}.
Different types of singular elliptic problems {involving Laplace operator} under mixed Dirichlet-Neumann boundary conditions can be found in \cite{MR2100910, MR1421215}.

With all the above literature as motivation, we are inquisitive about singular critical exponent problems with fractional Laplacian {defined via spectral decomposition} under mixed boundary conditions. We use Nehari manifold set and Gat\'{e}aux differentiability of energy functional(since $q\in (0,1)$) efficiently to tackle the singular power term.  Our problem lacks compactness in sense of Sobolev embedding due to an additional critical exponent term which makes \eqref{1} difficult as well as interesting. To overcome this, we seek the help of the Brezis-Lieb lemma and for establishing existence of second solution, we needed sharp estimates of minimizers of best constant in the inequality \eqref{sobolev-ineq}, given by \eqref{talenti}. Lemma \ref{new-1} is most important in this connection. Using these tools, we prove that \eqref{1} admits at least two weak solutions and our solutions are as regular as $L^\infty(\Omega)$. To the best of our findings, there is no literature dealing with multiplicity results to problems of the type \eqref{1}, so our article is  a lead to singular nonlocal problems in presence of mixed boundary conditions. We, further, remark that our work extends the existing  work of Mukherjee and Sreenadh \cite{MR3466525}, which was studied with fractional Laplacian through singular integral, in the Dirichlet-Neumann mixed boundary condition case. But we remark that our technique shall go well for \eqref{1} with singular integral definition of the fractional Laplacian under the functional setting adapted from \cite{MR3809115}.
We shall be interested to look upon \eqref{1} for $q>1$ case as a future direction.\\

The article is structured as follows: Section 2 presents the functional framework suitable for the spectral fractional Laplace operator with mixed boundary conditions. In this section, we also review the Caffarelli and Silvestre extension technique \cite{MR2354493}, which offers an alternative definition of the fractional Laplace operator using an auxiliary problem. In section 3 and section 4, we analyze fiber maps and Nehari manifold to prove existence of minimizers over suitable subsets of it. Section 5 is devoted to the main result where we show that the infimum of the energy functional over $\mathcal{N}_{\lambda}^{+}$ and  $\mathcal{N}_{\lambda}^{-}$ are achieved and are our desired weak solutions. The regularity  results have been proved in Section 6.

\section{Functional setting and Preliminaries}
\vspace{.5cm}
For the fractional Laplace operator, we recall its definition via spectral decomposition.  Let $(\phi_{j}, \lambda_{j})_j$ {denote} the eigenfunctions (standardised with respect to the $L^2(\Omega)$ norm) and the eigenvalues of $(-\Delta)$ with homogeneous mixed Dirichlet-Neumann boundary data.
Then, the pairs $(\phi_{j}, \lambda_{j}^{s})_j$ are the eigenfunctions and the corresponding eigenvalues of  the fractional Laplace operator $(-\Delta)^{s}$, with the same boundary conditions. Consequently, if
the functions $v_{j}$, $j=1,2$, are given by  $v_{j}(x)= \sum_{i\geq{1}} \langle v_{j},\phi_{i}\rangle \phi_{i}$, then we define
$$ \langle (-\Delta)^{s} v_1,v_2\rangle =\sum_{j\geq{1}} \lambda^s_{j}\langle v_1,\phi_j\rangle \langle v_2,\phi_j \rangle.$$
As a direct result of the preceding definition, we obtain
\begin{equation*}
(-\Delta)^{s} v_1 =\sum_{j\geq{1}} \lambda^s_{j}\langle v_1,\phi_{j}\rangle \phi_{j}.
\end{equation*}
In fact, we consider the spectral decomposition of any smooth function as
$$v=\sum_{j\geq{1}} b_{j}\phi_{j},\mbox{ where },b_{j}=\langle v,\phi_{j}\rangle \in \ell^2$$
and define the following norm
$${\|v\|_{H^{s}(\Omega)} = \left(\sum_{j\geq{1}}b^{2}_{j}{\lambda^{s}_{j}}\right)^{1/2}}<\infty.$$
Now we consider the Sobolev space ${H^{s}_{\sum_{D}}(\Omega)}$ which is defined as
$${H^{s}_{\sum_{D}}(\Omega)}= \overline{C^{\infty}_{0}(\Omega \cup~{\mathcal{N}})}^{\|.\|_{{H}^{s}(\Omega)}}$$
equipped with the norm
$$\|v\|_{H^{s}_{\sum_{D}}(\Omega)}=\|(-{\Delta)^{\frac{s}{2}}{v}\|_{L^{2}(\Omega)}} ,$$
The range $s \in (\frac{1}{2},1)$ is the appropriate range for mixed boundary problems due to the natural embedding of the associated functional space, see the following remark.
\begin{remark} \label{re1.1}
According to \cite{MR3393253}, it has already been emphasised that if $s\in (0,{1/2}]$ then $H^{s}_{0}(\Omega) = H ^{s}(\Omega)$ and consequently, $H^{s}_{\sum_{D}}(\Omega) =H ^{s}(\Omega)$, whereas for $s\in (1/2,1)$,  $H^{s}_{0}(\Omega) \subsetneq H^{s}(\Omega)$. So as a result, if $s\in(1/2,1)$, we obtain $H^{s}_{\sum_{D}}(\Omega) \subsetneq H^{s}(\Omega)$ which serves as an appropriate functional space to study the mixed boundary problem.\\
\end{remark}
 \begin{definition}
We say that a positive function $u\in H^{s}_{\sum_{D}}(\Omega)$ is a solution of \eqref{Main problem}  if
\begin{equation*}\label{Pl1}
    \int_{\Omega} (-\Delta)^\frac{s}{2} u (-\Delta)^\frac{s}{2} \phi dx = \int_{\Omega} (\lambda u^{-q}+ u^{2^*_s-1}) \phi dx
\end{equation*}
for all $\phi \in C^{\infty}_{0}(\Omega)$.
\end{definition}

By the preceding definition, we associate {to problem \eqref{1} the
corresponding} energy functional
$I : H^{s}_{\sum_{D}}(\Omega) \to (-\infty, \infty]$,   defined as
 \begin{equation*}\label{Pl}
I(u) = \frac12 \int_ {\Omega} |(-\Delta)^{\frac{s}{2} }u|^{2}  dx
- \lambda \int_\Omega  \frac{|u|^{1-q}}{1-q} dx
- \frac{1}{2^*_s} \int_\Omega |u|^{2^*_s}dx.
\end{equation*}
This functional $I$ is well defined in the space
$H^{s}_{\sum_{D}}(\Omega)$having critical points that correspond to solutions of \eqref{1}.

Our proof arguments depend on the techniques provided by Caffarelli and Silvestre in \cite{MR2354493} along with \cite{MR3023003, MR2646117}, in order to get a similar definition of  $(-\Delta)^{s}$ via an auxiliary problem that we
shall introduce below.

Given a domain $\Omega$, we denote the extension cylinder as $\mathcal{C}_{\Omega} = \Omega\times (0,\infty) \subset  \mathbb{R}^{N+1}_{+}$ and its lateral boundary as {$\partial_{L}\mathcal{C}_{\Omega} = {\partial {\Omega}} \times [0,\infty)$}.
 For a function $u \in H^{s}_{\sum_{D}}(\Omega)$,  its $s$-extension
 $\mathcal{E}_{s}: H^{s}_{\sum_{D}}(\Omega) \hookrightarrow \mathcal X^{s}_{\sum_{D}}(\mathcal{C}_{\Omega})$ is given by
 $$\mathcal{E}_{s}[u]= w,\quad u \in H^{s}_{\sum_{D}}(\Omega),$$
 where $w$ defined in the cylinder $\mathcal{C}_{\Omega}$ is the solution of the problem
\begin{equation}\label{ext-op}
\begin{cases}-\mbox{div}(y^{1-2s}\nabla w) = 0,&\mbox{in }\mathcal{C}_{\Omega},\\
\mathcal{A}^{*}(w) = 0,&\mbox{on }\partial_{L}\mathcal{C}_{\Omega},\\
w(x,0) = u(x),&\mbox{in }\Omega \times \{y=0 \},
\end{cases}\end{equation}
and here
$$\mathcal{A}^{*}(w)= w \mathcal{X}_{\sum^{*}_{D}}
+ \partial_\nu w \mathcal{X}_{ \sum^{*}_{\mathcal{N}}},$$
with {${\sum^{*}_{D}}={\sum_{D}}\times[0,{\infty})$~,~${\sum^{*}_\mathcal{N}}={\sum_\mathcal{N}}\times[0,{\infty})$. We use $\nu$ to denote the outward normal to $\partial \Omega$ as well as $\partial_L\mathcal{C}_\Omega$, with abuse of notation. But we remark that if $\nu$ denotes the outwards normal vector to $\partial \Omega$ and $\nu(x,y)$ as the outwards normal vector to $\partial \mathcal{C}_\Omega$ then, by
construction, $\nu(x,y)=(\nu,0), y>0$, refer \cite{MR2354493}.} The extension function $w$ belongs to the space
$ \mathcal X^{s}_{\sum_{D}}(\mathcal{C}_{\Omega},dxdy) = \overline{C^{\infty}_{0}((\Omega \cup{\sum_{\mathcal{N}}}~)\times[0,\infty)})^{\|.\|_{\mathcal{X}^{s}_{\sum_{D}}}(\mathcal{C}_{\Omega},dxdy)}$
which is a Hilbert space, with respect to the Hilbertian norm
\begin{equation}\label{norm-def}
\|f\|_{\mathcal{X}^{s}_{\sum_{D}} (\mathcal{C}_{\Omega},dxdy)} =
\left(\kappa_{s} \int_{\mathcal{C}_{\Omega}}y^{1-2s} |\nabla f(x,y)|^{2} dx dy
\right)^{1/2},
\end{equation}
where $\kappa_{s}$ is {a} suitable constant
(see for exact value of $\kappa_{s}$ in \cite{MR3023003}).

The spaces $H^{s}_{\sum_{D}}(\Omega)$ and $\mathcal{X}^{s}_{\sum_{D}}(\mathcal{C}_\Omega,dxdy)$ are isometric through the extension operator {$\mathcal{E}_{s}$, that is}
\begin{equation} \label{EO}
    \|\mathcal{E}_{s}[\phi]\|_{\mathcal{X}^{s}_{\sum_{D}}(C_\Omega,dxdy)}= \|\phi\|_{\mathcal{H}^{s}_{\sum_{D}}(\Omega)}
\end{equation}
for all $\phi \in {H}^{s}_{\sum_{D}}(\Omega)$.
The important feature of the extension function {$w$} is how it is connected to the original function $u$ via the fractional Laplacian by the formula
\[{\partial_{\nu^s} w}:= -\kappa_{s}  \lim_{y\to 0^{+}} y^{1-2s}
{\partial_y} w= (-\Delta)^{s}{u}.\]

\begin{remark}
For any $w\in \mathcal{X}^{s}_{\sum_{D}}(\mathcal{C}_\Omega,dxdy)$, from \cite{MR2354493,MR3023003} it is evident that trace of $w$ on $\partial \mathcal{C}_\Omega$ is well defined, since $s>1/2$. Moreover for any $\psi \in C^{\infty}_{0}((\Omega \cup{\sum_{\mathcal{N}}}~)\times[0,\infty))$, the following Green's type formula holds-
\begin{equation}\label{new1}
    -\int_{\mathcal{C}_\Omega} \psi~ div (y^{1-2s}\nabla w)~dxdy=
    \int_{\mathcal{C}_\Omega} y^{1-2s}\nabla w.\nabla \psi~dxdx - \int_{\partial\mathcal{C}_\Omega} \psi \left(y^{1-2s}~\frac{\partial w}{\partial \eta}\right)~dS.
\end{equation}
Thus $\langle \partial w/\partial \eta, \psi\rangle_{\partial \mathcal{C}_\Omega}$ can be defined in the spirit of \eqref{new1}, where $\langle .,\rangle_{\mathcal{C}_\Omega}$ denotes the duality pairing between $\mathcal{X}^{s}_{\sum_{D}}(\mathcal{C}_\Omega,dxdy)$ and $\big(\mathcal{X}^{s}_{\sum_{D}}(\mathcal{C}_\Omega,dxdy)\big)^*$. Now we can give sense to $\partial v /\partial \eta$ on $\partial \Omega$, for any $v\in H^s_{\sum_D}(\Omega)$ using
\[\frac{\partial v}{\partial \eta}{\bigg\vert}_{\partial \Omega}:= \frac{\partial (\mathcal{E}_s[v])}{\partial \eta}{\bigg\vert}_{\partial \mathcal{C}_\Omega}\]
which is well defined. For more details, we refer \cite{antil,lions2012non,MR2772469}.
\end{remark}

We recall that $N > 2s$ and the trace inequality proved in {\cite{MR4319042}}.
For any $r$, with $1\leq r\leq 2^*_s$, there exists a
constant $C=C(N,s, r,\Omega)>0$ such that
\begin{equation}\label{sobolev-ineq}
\int_{\mathcal{C}_{\Omega}}y^{1-2s} |\nabla f(x,y)|^{2} dx dy\geq C \Big(\int_{\Omega} |f(x,0)|^r dx \Big)^{2/r}
\end{equation}
for  any $f \in {\mathcal{X}^{s}_{0}(\mathcal{C}_\Omega,dxdy)}$.

When $r=2^*_s$,  the best constant in \eqref{sobolev-ineq}  is denoted by $S(s,N)$ and is independent of the domain $\Omega$. The exact value of $S(s,N)$ is given by
\[S(s,N)= \frac{2\pi^s\Gamma(1-s)\Gamma(\frac{N+2s}{2})(\Gamma(\frac{N}{2}))^{\frac{2s}{N}} }{\Gamma(s)\Gamma(\frac{N-2s}{2})(\Gamma(N))^s}\]
and $S(s,N)$ is not achieved in any bounded domain. Whereas in the whole space
$\mathbb R^N$ the inequality \eqref{sobolev-ineq} becomes
\begin{equation}\label{sobolev-ineq1}
\int_{\mathbb R^{N+1}_+}y^{1-2s} |\nabla f(x,y)|^{2} dx dy\geq S(s,N) \Big(\int_{\mathbb R^N} |f(x,0)|^{2^*_s} dx \Big)^{2/2^*_s}
\end{equation}
for any $f\in \mathcal{X}^{s}(\mathbb R^{N+1}_+):= {\overline{ C^\infty_0(\mathbb R^N\times (0,\infty))}^{\|\cdot\|_{\mathcal{X}^{s}(\mathbb R^{N+1}_+)}} }$, where $\|\cdot\|_{\mathcal{X}^{s}(\mathbb R^{N+1}_+)}$ is defined as in \eqref{norm-def} replacing $\mathcal C_{\Omega}$ by $\mathbb R^{N+1}_+$. The inequality
\eqref{sobolev-ineq1} is an equality for
$z= \mathcal E_s[u]$, where $u$ is given by
\begin{equation}\label{talenti}
   u(x)= u_{\epsilon}(x)= \frac{\epsilon^{\frac{N-2s}{2}}} (\epsilon^2 + |x|^2)^{\frac{N-2s}{2}},~~ x\in \mathbb{R}^N,
\end{equation}
with any arbitrary $\epsilon>0$. {We refer to} \cite{MR3023003,MR3912863} for more details.

\begin{remark}  If $s\in(\frac12,1)$ then following inclusions hold
\begin{equation}\label{inc}
\mathcal{X}^{s}_{0}(\mathcal{C}_{\Omega},dxdy)\subset \mathcal{X}^{s}_{\sum_{D}} (\mathcal{C}_{\Omega},dxdy)\subsetneq \mathcal{X}^{s} (\mathcal{C}_{\Omega},dxdy)
\end{equation}
i.e. the space  $\mathcal{X}^{s}_{0}(\mathcal{C}_{\Omega},dxdy)$ is subspace of $\mathcal{X}^{s} (\mathcal{C}_{\Omega},dxdy)$ and its  functions also vanish
on the lateral boundary $\partial_{L}\mathcal{C}_{\Omega}$ of $\mathcal{C}_{\Omega}$.
More specifically,
$$\mathcal{X}^{s}_{0}(\mathcal{C}_{\Omega}{,dxdy})= \bigg\{f\in L^2(\mathcal{C}_{\Omega}) : f=0 \mbox{ a.e. on }\partial_{L} \mathcal{C}_{\Omega}
\mbox and  \int_{\mathcal{C}_{\Omega}}y^{1-2s} |\nabla f(x,y)|^{2} dxdy<\infty\bigg\}.$$
\end{remark}

In the spirit of \eqref{ext-op}, we reformulate the original problem \eqref{1} in terms of the extension problem $(P_\lambda^*)$ below
\begin{equation*} \label{p}
\begin{cases}-\mbox{div}(y^{1-2s}{ \nabla }w) = 0,
&\mbox{in }\mathcal{C}_{\Omega},\\
\mathcal{A}^{*}(w) = 0,&\mbox{on }\partial_{L} \mathcal{C}_{\Omega},\\
{\partial_{\nu^s}w} = \lambda w^{-q}+ w^{2^*_s -1},\quad w>0,
&\mbox{in } \Omega \times \{y=0 \},
\end{cases}
\tag{$P^{*}_\lambda$}
\end{equation*}

\begin{definition}
 A function $w \in \mathcal{X}^{s}_{\sum_{D}}(\mathcal{C}_\Omega,dxdy)$, with $w>0$ in $ \Omega \times \{y=0 \}$,  is said to be a weak solution of  \eqref{p} {\rm if
\[\kappa_{s} \int_{\mathcal{C}_{\Omega}} y^{1-2s} \langle{\nabla w(x,y), \nabla \phi (x,y)} \rangle \,dxdy
 = \int_\Omega \big(\lambda  w(x,0)^{-q}  + w(x,0)^{2^*_s-1}\big)\phi(x,0) dx\]
 for all $\phi \in C^{\infty}_{0}(\mathcal{C}_\Omega{,dxdy})$.}
\end{definition}
{If $w \in \mathcal{X}^{s}_{\sum_{D}}{(\mathcal{C}_\Omega,dxdy)}$ is an energy or a weak solution of the problem $(P^{*}_{\lambda})$, then $u(x) = Tr[w](x) = w(x,0)$ belongs to the space  ${H}^{s}_{\sum_{D}}(\Omega)$ and is a weak solution of} problem \eqref{1}. {Also the} vice versa is true. That is
the $s$-extension $w = \mathcal{E}_{s}[u] \in \mathcal{X}^{s}_{\sum_{D}}(C_\Omega,dxdy)$ of any solution $u \in {{H}^{s}_{\sum_{D}}(\Omega)}$ of \eqref{1} is a solution of  $(P^{*}_{\lambda})$. Hence, both formulations of  problems $(P_\lambda)$ and $(P_\lambda^*)$
are equivalent.
Now, the energy functional associated to the problem $(P^{*}_{\lambda})$
is
\[
 J(w) = \frac{\kappa_{s}}{2} \int_{\mathcal{C}_{\Omega}} y^{1-2s} |\nabla w(x,y)|^2dxdy
- \frac{\lambda}{1-q} \int_\Omega |w(x,0)|^{1-q} dx
- \frac{1}{2^*_s} \int_\Omega |w(x,0)|^{2^*_s}dx,
\]
well defined for all $w\in~ \mathcal{X}^{s}_{\sum_{D}}(\mathcal{C}_{\Omega},dxdy)$.

The next theorems are main results of {the} paper.

\begin{theorem}\label{2}
 There exists a threshold $\Lambda>0$
 such that problem \eqref{1} possesses at least two weak solutions $u_\lambda^*$
and $v_\lambda^*$ in ${H}^{s}_{\sum_{D}}(\Omega)$, whenever $\lambda \in (0, \Lambda)$.
\end{theorem}

\begin{theorem}\label{3}
Any weak solution of \eqref{1} belongs to $L^\infty(\Omega)$.
\end{theorem}

\section{Analysis of Nehari manifold}
The features of the Nehari manifold $\mathcal{N}_{\lambda}$ associated to the functional $J$ is discussed in this section.
The Nehari manifold is defined as
 \[
\mathcal{N}_{\lambda} = \bigg\{ w\in \mathcal{X}^{s}_{\sum_{D}}(\mathcal{C}_\Omega{,dxdy}) ~ :~    \|w\|^2_{\mathcal{X}^{s}_{\sum_{D}}(\mathcal{C}_\Omega{,dxdy})} -\lambda  \int_{\Omega} |w(x,0)|^{1-q} dx
 -  \int_{\Omega}|w(x,0)|^{2^*_s} dx = 0\bigg\}.
\]
It is easy to see that $J$ fails to be bounded below over the whole space $\mathcal{X}^{s}_{\sum_{D}}(\mathcal{C}_\Omega{,dxdy})$, but the next result is of great utility.
\begin{theorem} \label{thm3.1}
$J$ is coercive and bounded below on $\mathcal{N}_{\lambda}$.
\end{theorem}
\begin{proof}
 Since $w\in \mathcal{N}_\lambda$, due to embedding of
 $\mathcal{X}^{s}_{\sum_{D}}(\mathcal{C}_\Omega{,dxdy})$  in $L^{2^*_{s}}(\Omega)$, we find
\begin{align*}
J(w) & = \left(\frac{1}{2}-\frac{1}{2^{*}_s} \right)\|w\|^2_{\mathcal{X}^{s}_{\sum_{D}}(\mathcal{C}_\Omega{,dxdy})}
- \lambda \left(\frac{1}{1-q}-\frac{1}{2^{*}_s}\right)\int_{\Omega} |w(x,0)|^{1-q}dx\\
& \geq A \|w\|^2_{\mathcal{X}^{s}_{\sum_{D}}(\mathcal{C}_\Omega{,dxdy})} - B \|w\|^{1-q}_{\mathcal{X}^{s}_{\sum_{D}}(\mathcal{C}_\Omega{,dxdy})}
\end{align*}
for some constants $A,B>0.$  Here, we use that 
$$\int_{\Omega} |w(x,0)|^{1-q}dx \leq C \left(\int_{\Omega} |w(x,0)|^{2^*_s}dx\right)^{\frac{1-q}{2^*_s}},$$ for some constant $C>0$.
 Hence, being $q\in(0,1)$, the functional $J$ is coercive. Indeed, $J(w) \to \infty$ as $\|w\|_{\mathcal{X}^{s}_{\sum_{D}}(\mathcal{C}_\Omega{,dxdy})} \to \infty$.
Define the map $Z : \mathbb{R}^{+} \to \mathbb{R}$ as
$$Z(t)= A_1t^2-A_2t^{1-q}.$$
Clearly, $Z$ has a unique point of minimum, since $1-q<2$.  This gives
at once that $J$ is bounded below on $\mathcal{N}_{\lambda}.$
\end{proof} \medskip

Now, in order to study the Nehari manifold geometry, let us  introduce
for all $w\in {\mathcal{X}^{s}_{\sum_{D}}(\mathcal{C}_{\Omega}, dxdy)}$ the associated fiber map $\Phi_w:\mathbb{R}^+ \to \mathbb{R}$, defined as
$\Phi_w(t)=J(tw)$ {for all $t\in\mathbb{R}^+$. Consequently,}
$$\Phi_w(t)=
		\frac{t^2}{2} \|w\|^2_{\mathcal{X}^{s}_{\sum_{D}}(\mathcal{C}_\Omega{,dxdy})} - \lambda \frac{t^{1-q}}{1-q} \int_{\Omega} |w(x,0)|^{1-q} dx
      - \frac{t^{2^*_s}}{2^*_s} \int_{\Omega} |w(x,0)|^{2^*_s} dx,
$$
so that
\begin{gather*}
\Phi'_w(t) = t \|w\|^2_{\mathcal{X}^{s}_{\sum_{D}}(\mathcal{C}_\Omega{,dxdy})} -\lambda t^{-q} \int_{\Omega} |w(x,0)|^{1-q} dx
 - {t^{2^*_s-1}} \int_{\Omega}|w(x,0)|^{2^*_s} dx , \\
 \Phi''_w(t) = \|w\|^2_{\mathcal{X}^{s}_{\sum_{D}}(\mathcal{C}_\Omega{,dxdy})} +  \lambda q t^{-q-1}
\int_{\Omega} |w(x,0)|^{1-q} dx - (2^*_s-1) t^{2^*_s-2}
\int_{\Omega}|w(x,0)|^{2^*_s} dx .
\end{gather*}
It is clear that the points in $\mathcal{N}_{\lambda}$ {correspond} to the critical points of $\Phi_{w}$ at $t = 1$. To represent {the local minima,
the local maxima and the points of inflexion of $\Phi_{w}$, the manifold} $\mathcal{N}_{\lambda}$ is splitted into three sets respectively. Accordingly, we define
\begin{align*}
\mathcal{N}_{\lambda}^{+}
&=  \{ w \in \mathcal{N}_{\lambda}:~ \Phi'_w (1) = 0,~ \Phi''_w(1) > 0\}\\
&=  \{ \Bar{t}w \in \mathcal{N}_{\lambda} :~ \Bar{t} > 0,~ \Phi'_w (\Bar{t}) = 0,~ \Phi''_w(\Bar{t}) > 0\},\\
\mathcal{N}_{\lambda}^{-}
&=  \{ w \in \mathcal{N}_{\lambda}:~ \Phi'_w (1) = 0,~~ \Phi''_w(1) <0\}\\
&=  \{ \Bar{t}w \in \mathcal{N}_{\lambda} : \Bar{t} > 0,~  \Phi'_w (\Bar{t}) = 0,~ \Phi''_w(\Bar{t}) < 0\}
\end{align*}
and $ \mathcal{N}_{\lambda}^{0}= \{ w \in \mathcal{N}_{\lambda} : \Phi'_{w}(1)=0,
 \Phi''_{w}(1)=0 \}$.

\begin{lemma}\label{ 3.2}
There exists $\lambda^*>0$ such that for all
$\lambda \in (0,\lambda^*)$ and for all  $w \in \mathcal{X}^{s}_{\sum_{D}}(\mathcal{C}_\Omega{,dxdy})$,
there are unique point $t_{\max}$, $t_*$ and $t^*$, with
$t_*<t_{\max}<t^*$, $t_* w\in \mathcal{N}_{\lambda}^{+}$ and
$ t^* w\in \mathcal{N}_{\lambda}^{-}$.
\end{lemma}

\begin{proof} Fix $w \in \mathcal{X}^{s}_{\sum_{D}}(\mathcal{C}_\Omega{,dxdy})$ and
 set, for convenience,
 $$M(w)= \int_{\Omega}|w(x,0)|^{1-q}dx\quad\mbox{and}\quad
N(w)= \int_{\Omega}|w(x,0)|^{2^*_s} dx.$$
Then,
\[
\frac{d}{dt}J(tw)
= t\|w\|^2_{\mathcal{X}^{s}_{\sum_{D}}(\mathcal{C}_\Omega{,dxdy})} - \lambda t^{-q} M(w) - t^{2^*_s-1} N(w)\\
= t^{-q} \big\{\mu_w(t) - \lambda M(w) \big\},
\]
where $\mu_w(t) := t^{1+q} {\|w\|^2_{\mathcal{X}^{s}_{\sum_{D}}(\mathcal{C}_{\Omega},dxdy)}} - t^{2^*_s-1+q} N(w)$.
Since $\mu_w(t) \to -\infty$ as  $t \to \infty$, the function $\mu_w$
 achieves its maximum at
\[
t_{\max} = \left[ \frac{(1+q)\|w\|^2_{\mathcal{X}^{s}_{\sum_{D}}(\mathcal{C}_\Omega{,dxdy})}}{(2^*_s-1+q) N(w)}
\right]^{\frac{1}{2^*_s-2}}
\]
 and the maximum of $\mu_w$ at the point $t_{\max}$ is given by
\[
\mu_w(t_{\max}) = \left( \frac{2^*_s-2}{2^*_s-1+q} \right)
\left( \frac{1+q}{2^*_s-1+q} \right)^{\frac{1+q}{2^*_s-2}}
\frac{(\|w\|^2_{\mathcal{X}^{s}_{\sum_{D}}(\mathcal{C}_\Omega{,dxdy})})^\frac{2^*_s-1+q}{2^*_s-2}}{(N(w))^\frac{1+q}{2^*_s-2}}.
 \]
Now, $tw \in \mathcal{N}_{\lambda}$  if and only if  $\mu_w(t) =  \lambda M(w) $
and we see that
\begin{align*}
\mu_w(t) - \lambda M(w)&\geq \mu_w(t_{\max}) - \lambda  \|w(x,0)\|^{1-q}_{{1-q}}\\
& =  \Big( \frac{2^*_s-2}{2^*_s-1+q} \Big)
\Big( \frac{1+q}{2^*_s-1+q} \Big)^{\frac{1+q}{2^*_s-2}}
\frac{(\|w\|^2_{\mathcal{X}^{s}_{\sum_{D}}(\mathcal{C}_\Omega{,dxdy})}
)^\frac{2^*_s-1+q}{2^*_s-2}}{{N(w)^\frac{1+q}{2^*_s-2}}}
- \lambda  \|w(x,0)\|^{1-q}_{{1-q}}>0,
\end{align*}
provided that
\[
\lambda<\lambda^*:=
\left( \frac{2^*_s-2}{2^*_s-1+q} \right)
\left( \frac{1+q}{2^*_s-1+q} \right)^{\frac{1+q}{2^*_s-2}}
(\|w\|^2_{\mathcal{X}^{s}_{\sum_{D}}(\mathcal{C}_\Omega{,dxdy})}
)^\frac{2^*_s-1+q}{2^*_s-2}
  C^{\frac{-1-q}{2^*_s-2}} (C_{1-q})^{-1},\]
 where {$C_{\alpha} $ for each $\alpha\geq 0$ is given by}
\begin{equation}
 C_{\alpha} = \sup \Big\{ \int_{\Omega} |w(x,0)|^{\alpha} dx :  w\in \mathcal{X}^{s}_{\sum_{D}}(\mathcal{C}_\Omega{,dxdy}), \;\|w\|_{\mathcal{X}^{s}_{\sum_{D}}(\mathcal{C}_\Omega{,dxdy})} = 1\Big\}.
 \end{equation}
\smallskip
An easy observation tells that $ \mu_w(t) = \lambda \int_{\Omega}  |w(x,0)|^{1-q}dx$ if and only if $\Phi'_w(t) = 0$. Therefore, whenever $\lambda \in (0,\lambda^*)$ there are exactly two points $t_*,t^*$ such that $0<t_*<t^*$ with $\mu'_w(t_*)>0$ and $\mu'_w(t^*)<0$ i.e. $t_*w \in \mathcal{N}^{+}_{\lambda}$ and $t^*w \in {\mathcal{N}^{-}_{\lambda}}$. Moreover,  $\Phi_{w}$ is decreasing in $(0,t_*)$ and increasing in  $(t_*,t^*)$ with a local minimum at $t=t_*$ and a local maximum at $t=t^*$.
 \smallskip

\end{proof} \medskip
\begin{corollary}\label{cor3.3}
It holds that $\mathcal{N}_{\lambda}^{0} = \{0\}$ for each $ \lambda \in (0, \lambda^*)$.
\end{corollary}

\begin{proof}
Assume by contradiction that $w \not\equiv 0 \in \mathcal{N}_{\lambda}^{0}$ and $\lambda\in (0,\lambda^*)$. Then, $w \in \mathcal{N}_{\lambda}$ that is, $1$ is a critical point of $\Phi_{w}$ by the definition of $\mathcal{N}_{\lambda}$.
We say that $\Phi_{w}$ possesses critical points corresponds to local minima or local maxima by Lemma \ref{ 3.2}. Therefore, 1 is the critical point of $\Phi_w$ corresponding to the local minima or local maxima of $\Phi_{w}$.
In view of this, either $w \in \mathcal{N}_{\lambda}^{+}$ or $w \in \mathcal{N}_{\lambda}^{-}$ that is a contradiction.
\end{proof} \medskip

From now on, we fix $ \lambda \in (0, \lambda^*)$,
without further mentioning. The next result establishes that $\mathcal{N}_{\lambda}^{+}$ and $\mathcal{N}_{\lambda}^{-}$ are bounded in a suitable sense.
\begin{lemma}\label{lem3.4}
 The following properties hold true.
\begin{itemize}
\item[$(P_{1})$] $\sup \{ \|w\|_{\mathcal{X}^{s}_{\sum_{D}}(\mathcal{C}_\Omega{,dxdy})}: w\in \mathcal{N}_{\lambda}^{+}\} < \infty $
\item[$(P_{2})$] $\inf \{ \|u\|_{\mathcal{X}^{s}_{\sum_{D}}(\mathcal{C}_\Omega{,dxdy})}: u \in \mathcal{N}_{\lambda}^{-} \} >0$  and
$ \sup \{ \|u\|_{\mathcal{X}^{s}_{\sum_{D}}(\mathcal{C}_\Omega{,dxdy})} : u \in \mathcal{N}_{\lambda}^{-} , J(u) \leq \mathcal{M}\} < \infty$,
for each $\mathcal{M} > 0$.
\end{itemize}
Furthermore, $\inf J(\mathcal{N}_{\lambda}^{+}) > - \infty$ and
$\inf J(\mathcal{N}_{\lambda}^{-}) > - \infty$.
\end{lemma}

\begin{proof}
$(P_{1})$:  {Fix} $w \in \mathcal{N}_{\lambda}^{+}$.
Then, by definition of $\mathcal{N}^+_{\lambda}$ and
the embedding results
\begin{align*}
0  < \Phi''_w(1) &=
(2-2^*_s) \|w\|^2_{\mathcal{X}^{s}_{\sum_{D}}(\mathcal{C}_\Omega{,dxdy})} + \lambda (2^*_s -1+q) \int_{\Omega}  |w(x,0)|^{1-q} dx   \\
 & \leq (2-2^*_s) \|w\|^2_{\mathcal{X}^{s}_{\sum_{D}}(\mathcal{C}_\Omega{,dxdy})} +  \lambda (2^*_s -1+q) C
  \|w(x,0)\|^{1-q}_{\mathcal{X}^{s}_{\sum_{D}}(\mathcal{C}_\Omega{,dxdy})} .
\end{align*}
This gives, for some constant $C>0$,
\[
\|w\|_{\mathcal{X}^{s}_{\sum_{D}}(\mathcal{C}_\Omega{,dxdy})} \leq \Big( \frac{\lambda (2^*_s -1+q) C} {2^*_s-2}
\Big)^{\frac{1}{1+q}}.
\]
\smallskip

\noindent
$(P_{2})$: Take $u \in \mathcal{N}_{\lambda}^{-}$. The definition of $\mathcal{N}^-_{\lambda}$ and the embedding results yield
\begin{align*}
0  > \Phi''_u(1)
&=   (1+q) \|u\|^2_{\mathcal{X}^{s}_{\sum_{D}}(\mathcal{C}_\Omega{,dxdy})} - (2^*_s-1+q) \int_{\Omega}|u(x,0)|^{2^*_s} dx  \\
 & \geq (1+q) \|u\|^2_{\mathcal{X}^{s}_{\sum_{D}}(\mathcal{C}_\Omega{,dxdy})} - (2^*_s-1+q) C'\|u(x,0)\|^{2^*_s}_{\mathcal{X}^{s}_{\sum_{D}}(\mathcal{C}_\Omega{,dxdy})}.
\end{align*}
Therefore, there exists a constant $C'>0$ such that
\[
\|u\|_{\mathcal{X}^{s}_{\sum_{D}}(\mathcal{C}_\Omega{,dxdy})} \geq \Big( \frac{1+q}{(2^*_s-1+q) C'}
 \Big)^{\frac{1}{2^*_s-2}}.
\]
This implies~  $\inf \{ \|u\|_{\mathcal{X}^{s}_{\sum_{D}}(\mathcal{C}_\Omega{,dxdy})} : u \in \mathcal{N}_{\lambda}^{-} \} >0$.
Now, if $J(u) \leq \mathcal{M}$, for some $\mathcal{M}>0$, {then
there exists a constant $C'>0$ such that}
\[
\frac{2^*_s-2}{2 2^*_s} \|u\|^2_{\mathcal{X}^{s}_{\sum_{D}}(\mathcal{C}_\Omega{,dxdy})}
- \lambda \Big( \frac{2^*_s-1+q}{2^*_s(1-q)} \Big)
C'  \|u\|^{1-q}_{\mathcal{X}^{s}_{\sum_{D}}(\mathcal{C}_\Omega{,dxdy})} \leq \mathcal{M}.
\]
Thus, we have proved that
for each $\mathcal{M}>0$
$$ \sup \{ \|u\|_{\mathcal{X}^{s}_{\sum_{D}}(\mathcal{C}_\Omega{,dxdy})}\, : \,u \in \mathcal{N}_{\lambda}^{-} ,\,\, J(u) \leq \mathcal{M}\} < \infty,$$
as stated.
Finally, property $(P_{1})$
shows that $\inf J(\mathcal{N}_{\lambda}^{+}) > - \infty$
and
property $(P_{2})$ implies
that $\inf J(\mathcal{N}_{\lambda}^{-}) > - \infty$.
\end{proof} \medskip
\section{Existence of {minimizers} over $\mathcal {N_{\lambda}}^{+}$ and $\mathcal {N_{\lambda}}^{-}$}

This section is devoted to the proof of the most crucial part of the article, i.e. to the proof that $\inf J(\mathcal N_\lambda^+)$ and $\inf J(\mathcal N_\lambda^-)$ are actually achieved.

\subsection{Properties of minimizers over $\mathcal N_{\lambda}^+$ and $\mathcal N_{\lambda}^-$}

This section contains essential properties of  minimizers of $J$ over $\mathcal N_\lambda^+$ and $\mathcal N_\lambda^-$, which {allow us to show that} these minimizers are (weak) solutions of $(P_\lambda^*)$.

\begin{lemma} \label{3.5}
Let $u_1$ and $u_2$ be two minimizers of $J$ over
$\mathcal{N}_{\lambda}^{+}$ and
$\mathcal{N}_{\lambda}^{-}$, respectively. Then for all non negative $   w \in \mathcal{X}^{s}_{\sum_{D}}(\mathcal{C}_\Omega{,dxdy})$
\begin{enumerate}
\item  there is $\epsilon^{'} > 0$ such that $J(u_1 +\epsilon w) \geq J(u_1)$
 for each $ \epsilon \in [0,\epsilon^{'}]$, and

\item  {if} $t_{\epsilon}$
 is the unique positive real number which satisfies
 $t_{\epsilon} (u_2 + \epsilon w) \in \mathcal{N}_{\lambda}^{-}$,
 {then} $t_{\epsilon} \to 1$  as $\epsilon \to 0^+$.
\end{enumerate}
\end{lemma}
\begin{proof} Fix  the two minimizers $u_1\in \mathcal{N}_{\lambda}^{+}$ and $u_2\in \mathcal{N}_{\lambda}^{-}$ of $J$ and also $ w \in \mathcal{X}^{s}_{\sum_{D}}(\mathcal{C}_\Omega,dxdy)$, as in the statement.\\
(1) Assume {that $w \geq 0$ a.e in $\mathcal{C}_\Omega$ and} set
$$
\sigma(\epsilon) = \|u_1+\epsilon w\|^2_{\mathcal{X}^{s}_{\sum_{D}}(\mathcal{C}_\Omega{,dxdy})} + \lambda q \int_{\Omega}
 |(u_1+\epsilon w)(x,0)|^{1-q}\,dx - (2^*_s-1) \int_{\Omega} |(u_1+\epsilon w)(x,0)|^{2^*_s} dx
$$
for all $\epsilon \ge 0$. By the continuity of $\sigma$ in $\mathbb R^+_0$ and
by the fact that
$\sigma(0) = \Phi''_{u_{1}}(1) >0$, being $u_1 \in \mathcal{N}^{+}_{\lambda}$, there exists
$\epsilon^{'} >0 $ such that $\sigma(\epsilon)> 0$ for $\epsilon \in [0, \epsilon^{'}]$.
 For all $\epsilon > 0$, there exists a unique $t^*_{\epsilon}>0$ with
$t^*_{\epsilon}(u_1 + \epsilon w) \in \mathcal{N}^{+}_{\lambda}$, by Lemma \ref{ 3.2}, and $t^*_{\epsilon} \to 1$
as $\epsilon \to 0$, since $u_1 \in \mathcal{N}^+_{\lambda}$. For each $\epsilon \in [0, \epsilon^{'}]$ we have
\[
J(u_1 + \epsilon w) \geq J(t^*_{\epsilon}(u_1 + \epsilon w))
\geq \inf J(\mathcal{N}^{+}_{\lambda})= J(u_1),
\]
as required.
\smallskip

\noindent
(2)  Define the map  $g :{\mathbb R^+}\times \mathbb{R}^3 \to \mathbb{R}  $ by
\[
g(t,a_1,a_2,a_3) = a_1t - \lambda t^{-q}a_2 - t^{2^*_s-1}a_3
\]
for $(t,a_1,a_2,a_3)\in \mathbb R^+\times \mathbb{R}^3$.
Clearly, $g$ is a $C^{\infty}$ function and we find
\begin{equation*}
\frac{dg}{dt} \Big(1, \|u_2\|^2_{\mathcal{X}^{s}_{\sum_{D}}(\mathcal{C}_{\Omega}{,dxdy})},\int_{\Omega}|u_2(x,0)|^{1-q}\,dx ,
\int_{\Omega}|u_2(x,0)|^{2^*_s}dx\Big) = \Phi''_{u_2}(1)<0,
\end{equation*}
since $ u_2 \in \mathcal{N}^-_{\lambda}$.
For each $\epsilon > 0$ we have
$$g \big(t_{\epsilon},
\|u_2+\epsilon w\|^2_{\mathcal{X}^{s}_{\sum_{D}}(\mathcal{C}_\Omega{,dxdy})},  \int_{\Omega}|(u_2 + \epsilon w)(x,0)|^{1-q}\,dx,
\int_{\Omega} |u_2(x,0)|^{2^*_s} dx\big) = \Phi'_{u_2+\epsilon w}(t_\epsilon)=0.$$
 Furthermore,
\[
g\Big(1, \|u_2\|^2_{\mathcal{X}^{s}_{\sum_{D}}(\mathcal{C}_\Omega{,dxdy})}, \int_{\Omega} |u_2(x,0)|^{1-q}\,dx, \int_{\Omega} |u_2(x,0)|^{2^*_s}dx\Big)
= \Phi'_{u_2}(1) = 0.
\]
As a consequence of the implicit function theorem, there exists open neighborhoods
$ U \subset {\mathbb R^+}$, $U^{'} \subset \mathbb{R}^3$ containing
$1$ and $\Big(\|u_2\|^2_{\mathcal{X}^{s}_{\sum_{D}}(\mathcal{C}_\Omega{,dxdy})},\int_{\Omega} |u_2(x,0)|^{1-q}\,dx, \int_{\Omega}|u_2(x,0)|^{2^*_s}dx \Big)$,
respectively, such that for each $y \in U^{'} $ the equation $g(t,y) = 0 $ has a unique
solution $ t = l(y)\in U $ {and  $l : U^{'} \to U$ is} a continuous function.
Hence,
$$\Big(\|u_2+\epsilon w\|^2_{\mathcal{X}^{s}_{\sum_{D}}(\mathcal{C}_\Omega{,dxdy})},\; \int_{\Omega}|(u_2+ \epsilon w)(x,0)|^{1-q}\,dx,
\int_{\Omega}|(u_2+ \epsilon w)(x,0)|^{2^*_s}dx\Big) \in U^{'}$$
and
\[
 l \left( \|u_2+\epsilon w)\|^2_{\mathcal{X}^{s}_{\sum_{D}}(\mathcal{C}_\Omega{,dxdy})},\; \int_{\Omega}|(u_2+ \epsilon w)(x,0)|^{1-q}\,dx,
 \int_{\Omega}|(u_2+ \epsilon w)(x,0)|^{2^*_s}dx \right) = t_{\epsilon},
\]
{being} $g\Big(t_{\epsilon},\|u_2+\epsilon w)\|^2_{\mathcal{X}^{s}_{\sum_{D}}(\mathcal{C}_\Omega{,dxdy})}$,
$\int_{\Omega}|(u_2+ \epsilon w)(x,0)|^{1-q}\,dx$,
$\int_{\Omega}|(u_2+ \epsilon w)(x,0)|^{2^*_s}dx\Big) = 0$. Therefore,
{$t_{\epsilon} \to 1$ as $\epsilon \to 0^+$  by the} continuity of $l$.
 \end{proof} \medskip

{\begin{remark}
    Due to the definition of $J$, it is easy to note that $J(u)=J(|u|)$, for $u\in \mathcal{X}^{s}_{\sum_{D}}(\mathcal{C}_\Omega{,dxdy})$. This allows us to assume that any minimizer of $J$ over $\mathcal{N}_{\lambda}^{+}$ and $\mathcal{N}_{\lambda}^{-}$ are non negative.
\end{remark}}

\begin{lemma} \label{lem3.6}
Assume that $u_1$ and $u_2$ are non negative minimizers of $J$ on $\mathcal{N}_{\lambda}^{+}$ and
$\mathcal{N}_{\lambda}^{-}$ respectively. Then
for all $w\in \mathcal{X}^{s}_{\sum_{D}}(\mathcal{C}_\Omega{,dxdy})$, we have $u^{-q}_1(.,0)w(.,0) ~,~ u^{-q}_2(.,0) w(.,0) \in L^{1}(\Omega)$ and
\begin{gather}
\kappa_{s} \int_{\mathcal{C}_{\Omega}} y^{1-2s} \langle{\nabla u_1, \nabla w } \rangle \,dx\,dy
 - \lambda \int_\Omega  u_1^{-q}(x,0)w(x,0)\,dx
 - \int_\Omega u_1^{2^*_s-1}(x,0)w(x,0) dx  \geq 0, \label{eq4.2}\\
\kappa_{s} \int_{\mathcal{C}_{\Omega}} y^{1-2s} \langle{\nabla u_2, \nabla w } \rangle \,dx\,dy
 - \lambda \int_\Omega  u_2^{-q}(x,0)w(x,0)\,dx
 - \int_\Omega u_2^{2^*_s-1}(x,0)w(x,0) dx \geq 0.\label{eq4.3}
\end{gather}
\end{lemma}
\begin{proof}
Fix  the two {non negative} minimizers $u_1\in \mathcal{N}_{\lambda}^{+}$ and $u_2\in \mathcal{N}_{\lambda}^{-}$ of $J$ and also $ w \in \mathcal{X}^{s}_{\sum_{D}}(\mathcal{C}_\Omega,dxdy)$, as in the statement.
Then for   $\epsilon > 0$ small enough, Lemma \ref{3.5} implies that
\begin{align}
0 &\leq \frac{J(u_1+\epsilon w) - J(u_1)}{\epsilon}\nonumber\\
 &=  \frac{1}{2\epsilon} \left(\|u_1+\epsilon w\|^2_{\mathcal{X}^{s}_{\sum_{D}}(\mathcal{C}_\Omega{,dxdy})}
 - \|u_1\|^2_{\mathcal{X}^{s}_{\sum_{D}}(\mathcal{C}_\Omega{,dxdy})}\right)
 \label{eq7}\\
&\quad  - \frac{\lambda}{\epsilon} \int_{\Omega} \left(|(u_1 + \epsilon w)(x,0)|^{1-q}- |u_1(x,0)|^{1-q}\right)dx
 - \frac{1}{\epsilon 2^*_s} \int_{\Omega} \left(|(u_1+\epsilon w)(x,0)|^{2^*_s} - |u_1(x,0)|^{2^*_s}\right) dx .\nonumber
\end{align}

An easy computation leads to
\begin{enumerate}
\item[(i)]
$\displaystyle{
 \lim_{\epsilon \to 0^+} \Bigg(\frac{\|u_1+ \epsilon w\|^2_{\mathcal{X}^{s}_{\sum_{D}}(\mathcal{C}_\Omega{,dxdy})}- \|u_1\|^2_{\mathcal{X}^{s}_{\sum_{D}}(\mathcal{C}_\Omega{,dxdy})}}{\epsilon}\Bigg)
=   2 \kappa_{s} \int_{\mathcal{C}_{\Omega}} y^{1-2s} \langle{\nabla u_1,\nabla w} \rangle \,dx\,dy},$
\item[(ii)]
$\displaystyle{
 \lim_{\epsilon \to 0^+} \int_{\Omega} \frac{\big(|(u_1+\epsilon w)(x,0)|^{2^*_s}-|u_1(x,0)|^{2^*_s}\big)}{\epsilon} dx
= 2^*_s \int_{\Omega} |u_1(x,0)|^{2^*_s-1} w(x,0) dx}$
\end{enumerate}
and from \eqref{eq7}, we get
 $$\frac{1}{\epsilon}\int_{\Omega} \Big( {|(u_1+\epsilon w)(x,0)|^{1-q} - |u_1(x,0)|^{1-q}} \Big) dx < \infty.$$
For each $x\in \Omega$, the quantity
\[\frac{1}{\epsilon}\left(\frac{{|(u_1+\epsilon w)(x,0)|^{1-q} - |u_1(x,0)|^{1-q}} }{1-q}\right)\]
is monotonically increasing as $\epsilon \downarrow 0$ and
$$
\lim_{\epsilon \downarrow 0}
\frac{|(u_1+\epsilon w)(x,0)|^{1-q} - |u_1(x,0)|^{1-q}}{\epsilon}
=\begin{cases}
	0 & \text{if } w(x,0)=0 \\
	u_1^{-q}(x,0) w(x,0) & \text{if }  w(x,0)>0 , u_1(x,0) > 0\\
	\infty & \text{if } w(x,0) > 0 ,\; u_1(x,0) =0.
    \end{cases}
$$
Thus, applying the monotone convergence theorem, we get $(u_1(x,0))^{-q}w(x,0) \in L^1(\Omega)$.
Supposing $\epsilon \downarrow 0$ on both sides of \eqref{eq7},
we get \eqref{eq4.2}.

Next, we shall prove all these properties for $u_2$. For each $\epsilon >0 $,
there exists $t_{\epsilon}>0$ such that $t_{\epsilon}(u_2+\epsilon w) \in \mathcal{N}^-_\lambda$, by {Lemma} \ref{ 3.2}.
Now  using Lemma \ref{3.5}-(2), for small enough $\epsilon > 0$ it holds
\[
J(t_{\epsilon}(u_2+\epsilon w)) \geq J(u_2) \geq J(t_{\epsilon}u_2).
\]
Hence $J(t_{\epsilon}(u_2+\epsilon w)) - J(t_{\epsilon}u_2) \geq 0$ and we get
\begin{align*}
 \lambda \int_{\Omega}  \Big(|(u_2+\epsilon w)(x,0)|^{1-q} - |u_2(x,0)|^{1-q}\Big) dx
&\leq \frac{t_{\epsilon}^{q}}{2} \Big(\|u_2+\epsilon w\|^2_{\mathcal{X}^{s}_{\sum_{D}}(\mathcal{C}_\Omega{,dxdy})}
- \|u_2\|^2_{\mathcal{X}^{s}_{\sum_{D}}(\mathcal{C}_\Omega{,dxdy})}\Big)\\
&\; - \frac{t_{\epsilon}^{2^*_s+q}}{2^*_s}
\int_{\Omega}\Big(|(u_2+\epsilon w)(x,0)|^{2^*_s} - |u_2(x,0)|^{2^*_s}\Big) dx.
\end{align*}
Since $t_\epsilon \to 1$ as $\epsilon \downarrow 0$, {using the } same arguments as above, we get {that} $u_2(.,0)^{-q}w(.,0) \in L^1(\Omega)$
and the required inequality \eqref{eq4.3}.
\end{proof} \medskip

Suppose {that} $\phi_{1,s}>0$ is the first eigenfunction of $(-\Delta)^s$ with homogeneous Dirichlet boundary condition,that is
$$\begin{cases}
    (-\Delta)^s \phi_{1,s} = \lambda_1^s \phi_{1,s}&\text{in}\; \Omega,\\
    \phi_{1,s} =0&{\text{in}\; \mathbb R^N\setminus \Omega}.
\end{cases}$$
Then, $0<\phi_{1,s}\in L^{\infty}(\Omega) \cap \mathcal{C}^\alpha(\overline{\Omega})$ for some $\alpha\in(0,1)$,
as shown in \cite{MR3023003}.
Its $s$-extension defined by $\Psi_{1,s}=\mathcal{E}_s[\phi_{1,s}] \in  L^{\infty}(\mathcal{C}_\Omega) \cap \mathcal{C}^\alpha(\overline{\mathcal{C}_\Omega})$ is associated  to the following problem in the spirit of \eqref{ext-op},
\begin{equation*}
\begin{cases} -\mbox{div}(y^{1-2s}{\nabla\Psi_{1,s}}\ ) = 0 &\mbox{ in }~ \mathcal{C}_{\Omega}\\
                       \Psi_{1,s} = 0~ &\mbox{ on }~{\partial_{L}} \mathcal{C}_{\Omega},\\
                      \dfrac{\partial \Psi_{1,s}}{\partial \nu^s} = \lambda_1^s\Psi_{1,s} &\mbox{ in }~ \Omega \times \{y=0 \} \end{cases}
\end{equation*}
and satisfies for all $0\leq \varphi \in C_0^\infty({\mathcal{C}_{\Omega},dxdy})$
\begin{equation}\label{efn-rel}
    \kappa_{s} \int_{\mathcal{C}_{\Omega}} y^{1-2s} \langle{\nabla \Psi_{1,s}(x,y), \nabla \varphi (x,y)} \rangle \,dxdy
 = \lambda_1^s\int_\Omega \Psi_{1,s}(x,0)\varphi(x,0) \,dx.
\end{equation}
Let $\beta > 0$ be so small that
\begin{equation}\label{def-f}
    f = \beta \Psi_{1,s}
\end{equation}
satisfies for all $0\leq \varphi \in C_0^\infty({\mathcal{C}_{\Omega},dxdy})$
 \begin{equation} \label{eq3.5}
-\kappa_{s} \int_{\mathcal{C}_{\Omega}} y^{1-2s} \langle{\nabla f, \nabla \varphi} \rangle \,dx\,dy
+\int_\Omega  (\lambda f^{-q}(x,0)+f^{2^*_s-1}(x,0))\varphi(x,0) dx  >0,
\end{equation}
 and  $f^{2^*_s-1+q} (x) \leq \lambda{q/(2^*_s-1)}$
for each $x \in \Omega$. We justify the choice of such $\beta>0$ using $0<\Psi_{1,s}\in L^{\infty}(\mathcal{C}_\Omega)$ implies for each $x\in \Omega$, $\Psi_{1,s}(x,0)\leq c ,~\text{for some constant}~  c>0$. This says that
\begin{equation}\label{e3.6}
 \Bigg(\frac{\lambda q}{(2^*_s-1)(\Psi_{1,s}(x,0))^{2^*_s-1+q}}\Bigg)^{\frac{1}{2^*_s-1+q}} \geq \Bigg(\frac{\lambda q}{(2^*_s-1) c^{2^*_s-1+q}}\Bigg)^{\frac{1}{2^*_s-1+q}} :=\eta~ .
\end{equation}
Hence we deduce that for each $x\in \Omega$
$$(\eta \Psi_{1,s}(x,0))^{2^*_s-1+q} =\frac{\lambda q}{(2^*_s-1) c^{2^*_s-1+q}} \Psi_{1,s}^{2^*_s-1+q} \leq \frac{\lambda q}{2^*_s-1 }.$$
This implies the claim for all  $\beta<\eta$. Hence $f^{2^*_s-1+q} (x) \leq \lambda {q/(2^*_s-1)}$ holds for $\beta <\eta$, whereas \eqref{eq3.5} holds true as $\beta \to 0^+$. Thus we can choose    $\beta>0$ so small that \eqref{eq3.5} and \eqref{def-f} both holds together.

\begin{lemma} \label{lem2.6}
For each non negative $w \in \mathcal{X}^{s}_{\sum_{D}}({\mathcal{C}_{\Omega},dxdy})$, there exists
{a sequence $(w_m)_m$} in $\mathcal{X}^{s}_{\sum_{D}} ({\mathcal{C}_{\Omega},dxdy})$
such that each $w_{m}$ is compactly supported in $\mathcal{C}_\Omega$, with $0 \leq w_1 \leq w_2 \leq \ldots$  and $w_m \to w$ strongly in $\mathcal{X}^{s}_{\sum_{D}}({\mathcal{C}_{\Omega},dxdy}) $ as $m\to \infty$.
\end{lemma}
\begin{proof}
Let $0\leq w \in \mathcal{X}^{s}_{\sum_{D}} ({\mathcal{C}_{\Omega},dxdy})$ then by density argument we know there exists a  ${(g_m)_m} \subset {\mathcal{C}}^{\infty}_{c}(\mathcal{C}_\Omega)$
such that $g_m$ are non negative and $g_m \to w$  strongly in $\mathcal{X}^{s}_{\sum_{D}}({\mathcal{C}_{\Omega},dxdy})$ as $m\to \infty$.
Define $u_m =  \min \{g_m, w\}$ for each $m$ which strongly converges to $w$ in
$\mathcal{X}^{s}_{\sum_{D}} ({\mathcal{C}_{\Omega},dxdy})$. Let $w_1 = u_{s_1}$, choosing $s_1\in \mathbb N$ which satisfies
\[\|u_{s_1}-w\|_{\mathcal{X}^{s}_{\sum_{D}} ({\mathcal{C}_{\Omega},dxdy})} \leq 1.\]
Then $\max \{w_1, u_m\} \to w$ strongly in $\mathcal{X}^{s}_{\sum_{D}} ({\mathcal{C}_{\Omega},dxdy})$ as
$m \to \infty$.  We now choose $s_2\in \mathbb N$ satisfying
\[\|\max\{w_1,u_{s_2}\}-w\|_{\mathcal{X}^{s}_{\sum_{D}} ({\mathcal{C}_{\Omega},dxdy})} \leq 1/2.\]
Let $w_2=\max\{w_1,u_{s_2}\}$ then $\max \{w_2, u_m\} \to w$ strongly as $m \to \infty$.
Proceeding as above, in general we set
\[w_{m+1}= \max\{w_m,u_{s_{m+1}}\}.\]
It is easy to see that $w_m \in \mathcal{X}^{s}_{\sum_{D}} ({\mathcal{C}_{\Omega},dxdy})$  has compact support
for each $m$ and
\[\|\max\{w_m,u_{s_{m+1}}\}-w\|_{\mathcal{X}^{s}_{\sum_{D}} ({\mathcal{C}_{\Omega},dxdy})} \leq \frac{1}{m+1}\]
which implies that
{ $(w_m)_m$} converges strongly to $w$ in $\mathcal{X}^{s}_{\sum_{D}} ({\mathcal{C}_{\Omega},dxdy})$ as $m \to \infty$.
\end{proof} \medskip

With the aid of this, we prove our next crucial lemma.

\begin{lemma}\label{le05}
Assume $u_1$ and $u_2$ are minimizers of $J$ on $\mathcal{N}_{\lambda}^{+}$ and
$\mathcal{N}_{\lambda}^{-}$ respectively. Then there holds $u_1 \geq f$ and $u_2\geq f$ a.e. in $\mathcal{C}_\Omega$, where $f$ is defined in \eqref{def-f}.
 \end{lemma}
\begin{proof}
An easy noticeable fact is

\begin{equation}\label{eq3.4}
\frac{d}{dr}(\lambda r^{-q}+r^{2^*_s-1})
= -q \lambda r^{-q-1} +(2^*_s-1)r^{2^*_s-2} \leq 0
\end{equation}
if and only if $ r^{2^*_s-1+q}\leq  \frac{\lambda q}{2^*_s-1}$.
By {Lemma} \ref{lem2.6}, we get existence of a non negative sequence
$(w_k)_k$ in $\mathcal{X}^{s}_{\sum_{D}}(\mathcal{C}_\Omega{,dxdy})$  such that $0\leq w_k\leq (f-u_1)^+$, each $w_k$ is compactly supported in $\mathcal{C}_\Omega$ and as $k\to \infty$
\[w_k\to (f-u_1)^+\;\text{strongly in } \mathcal{X}^{s}_{\sum_{D}}(\mathcal{C}_\Omega{,dxdy}).\]
Thus,   \eqref{eq4.2} and \eqref{eq3.5} imply that
\begin{align*}
 &\kappa_{s} \Big (\int_{\mathcal{C}_{\Omega}} y^{1-2s} \langle{\nabla u_1, \nabla w_{k}} \rangle \,dx\,dy
- \lambda \int_\Omega  u_1^{-q}(x,0)w_k(x,0)\,dx - \int_\Omega u_1^{2^*_s-1}(x,0) w_k(x,0) dx \Big)\\
 & -\kappa_{s} \Big ( \int_{\mathcal{C}_{\Omega}} y^{1-2s} \langle{\nabla f, \nabla w_{k}} \rangle \,dx\,dy
- \lambda \int_\Omega  f^{-q}(x,0)w_k(x,0)\,dx
- \int_\Omega f^{2^*_s-1}(x,0)w_k(x,0) dx \Big) \geq 0,
\end{align*}
which yields
\begin{align*}
&\kappa_{s} \int_{\mathcal{C}_{\Omega}} y^{1-2s}\Big ( \langle{\nabla u_1, \nabla w_{k}} \rangle - \langle{\nabla f, \nabla w_{k}} \rangle\Big ) dx\,dy \\
&- \int_\Omega \Big(\lambda u_1^{-q}(x,0)  + u_1^{2^*_s-1}(x,0)\Big)w_k(x,0)\,dx
+ \int_{\Omega} \Big(\lambda f^{-q}(x,0) +f^{2^*_s-1}(x,0)\Big)w_k(x,0)\,dx \geq 0.
\end{align*}
Since obviously $w_k(x,0)\to (f- u_1)^+(x,0)$ pointwise a.e. in $\Omega$, we set
$w_k(x,0) = (f- u_1)^+(x,0)+ o(1)$ as $k \to \infty$. Now considering the first term, we see that
\begin{align*}
&\kappa_{s} \Big (\int_{\mathcal{C}_{\Omega}} y^{1-2s} \langle{\nabla (u_1-f), \nabla w_{k}} \rangle \,dx\,dy \Big) = \kappa_{s} \Big (\int_{\mathcal{C}_{\Omega}} y^{1-2s} \langle{\nabla (u_1-f), \nabla(f-u_1)^+  } \rangle \,dx\,dy \Big) .
\end{align*}
Also $$\nabla(f -u_1)(x,0)=\nabla((f-u_1)^{+}(x,0)-(f-u_1)^{-}(x,0)).$$
Thus, we obtain
\begin{align*}
& \kappa_{s} \Big (\int_{\mathcal{C}_{\Omega}} y^{1-2s} \langle{\nabla (u_1-f), \nabla(f-u_1)^{+}} \rangle \,dx\,dy \Big)\\
 &=\kappa_{s} \Big (\int_{\mathcal{C}_{\Omega}} y^{1-2s} \langle{\nabla (f-u_1)^{-}, \nabla(f-u_1)^{+}\rangle\,dx\,dy} - \kappa_{s} \Big (\int_{\mathcal{C}_{\Omega}} y^{1-2s} \langle{\nabla (f-u_1)^{+}}, \nabla(f-u^{+}_1) \rangle \,dx\,dy \Big) \\
&= - \| (f - u_1)^{+}\|^{2}_{\mathcal{X}^{s}_{\sum_{D}}(\mathcal{C}_\Omega{,dxdy})}.
\end{align*}
Since $f^{2^*_s-1+q} (x) \leq  \frac{\lambda q}{2^*_s-1}$
for all $x \in \Omega$, using \eqref{eq3.4} we get
\begin{align*}
& \int_\Omega \Big((\lambda u^{-q}_1(x,0) + u^{2^*_s-1}_1(x,0))
 - (\lambda f^{-q}(x,0) + f^{2^*_s-1}(x,0))\Big)w_k(x,0)\,dx \\
& = \int_{\Omega \cap \{f \geq u_1\}} \Big((\lambda u^{-q}_1(x,0) + u^{2^*_s-1}_1(x,0))
 - (\lambda f^{-q}(x,0) + f^{2^*_s-1}(x,0))\Big)((f-u_1)^{+}(x,0)+o(1))\,dx  \\
& \geq 0.
 \end{align*}
Altogether we have
 \begin{align*}
 0 &\leq -\|(f-u_1 )^+\|^2_{\mathcal{X}^{s}_{\sum_{D}}(\mathcal{C}_\Omega{,dxdy})} -\int_\Omega (\lambda (u^{-q}_1(x,0) + u^{2^*_s-1}_1(x,0))w_k(x,0)\,dx\\
&\quad + \int_{\Omega} (\lambda f^{-q}(x,0) + f^{2^*_s-1}(x,0))w_k(x,0)\,dx +o(1)\\
  &\leq -\|(f-u_1 )^+\|^2_{\mathcal{X}^{s}_{\sum_{D}}(\mathcal{C}_\Omega{,dxdy})} + o(1)
  \end{align*}
As $k \to \infty$, this gives
$\|(f-u_1 )^+\|^2_{\mathcal{X}^{s}_{\sum_{D}}(\mathcal{C}_\Omega{,dxdy})} \leq 0$ which says that $u_1\geq f$ a.e. in $\mathcal{C}_\Omega$. The same argument shows that $u_2\geq f$ a.e. in $\mathcal{C}_\Omega$.
\end{proof} \medskip

\section{Proof of main result}
This section consists of the proof of the existence of $u_\lambda,\;v_\lambda$, which {are} minimizers of $J$ over $\mathcal N_\lambda^+$ and $\mathcal N_\lambda^-$, respectively, as mentioned in
{the} earlier section. Later on, we show that the minimizers
$u_\lambda,\;v_\lambda$ forms two weak solution of $(P_\lambda^*)$.

\subsection{Existence of first solution}

In this first subsection, we will show that the infimum of $J$ over $\mathcal{N}_{\lambda}^{+}$ is achieved.
\begin{lemma}\label{new-1}
It holds that $\inf_{w\in \mathcal{N}_{\lambda}^{+}} J(w) < 0$.
\end{lemma}
\begin{proof}
Fix $w_0 \in \mathcal{N}_{\lambda}^{+}$.
{Hence,} $\Phi''_{w_0}(1) >0$ i.e.
\[
\Big(\frac{1+q}{2^*_s-1+q} \Big)\|w_0\|^2_{\mathcal{X}^{s}_{\sum_{D}}(\mathcal{C}_\Omega{,dxdy})}
> \int_{\Omega}|w_0(x,0)|^{2^*_s} dx .
\]
Since $2^{*}_{s}>2$, we get at once that $J(w_0)<0$. Indeed,
\begin{align*}
J(w_0)
& = \Big( \frac12 - \frac{1}{1-q} \Big) \|w_0\|^2_{\mathcal{X}^{s}_{\sum_{D}}(\mathcal{C}_\Omega{,dxdy})}
+ \Big( \frac{1}{1-q} - \frac{1}{2^*_s} \Big)\int_{\Omega} |w_0(x,0)|^{2^*_s}\\
& \leq -  \frac{1+q}{2(1-q)} \|w_0\|^2_{\mathcal{X}^{s}_{\sum_{D}}(\mathcal{C}_\Omega{,dxdy})} + \frac{1+q}{2^*_s(1-q)} \|w_0\|^2_{\mathcal{X}^{s}_{\sum_{D}}(\mathcal{C}_\Omega{,dxdy})}\\
&= \Big( \frac{1}{2^*_s}-\frac{1}{2}\Big)\Big(\frac{1+q}{1-q}\Big) \|w_0\|^2_{\mathcal{X}^{s}_{\sum_{D}}(\mathcal{C}_\Omega{,dxdy})}<0.
\end{align*}
This proves  $\inf_{w\in \mathcal{N}_{\lambda}^{+}} J(w) < 0$.
\end{proof} \medskip

Let us define
\begin{align*}
     \Lambda := \sup\Bigg\{\lambda>0:\; \text{for each}\; &u\in {\mathcal{X}_{\sum_D}^s(\mathcal{C}_\Omega,dxdy)},\; \Phi_u \;\text{has two critical points and }\\
     & \sup\{\|u\|_{{\mathcal{X}_{\sum_D}^s(\mathcal{C}_\Omega,dxdy)} }: \; u\in \mathcal{N}_\lambda^{+}\} \leq  \Big(\frac{2^*_{s}}{2}\Big)^{\frac{2}{2^*_{s}-2}} S(s,N)^{{\frac{2^*_{s}}{2^*_{s}-2}}}\Bigg\}
     \end{align*}
     then $\Lambda>0$ due to {Corollary} \ref{cor3.3} and {Lemma} \ref{lem3.4}.
\begin{proposition}\label{u-lambda}
 There exists $u_\lambda \in \mathcal{N}_{\lambda}^{+}$,
with $J(u_\lambda) = \inf_{w\in \mathcal{N}_{\lambda}^{+}} J(w)$ for all $ \lambda\in (0,\Lambda)$.
\end{proposition}
\begin{proof}
Let  $(w_{k})_{k} \subset \mathcal{N}_{\lambda}^{+}$ be a sequence such that $J(w_{k}) \to \inf_{w\in \mathcal{N}_{\lambda}^{+}} J(w)$ as $k \to \infty$.   {Lemma \ref{lem3.4} yields} that there {exists} $u_\lambda \in \mathcal{X}^{s}_{\sum_{D}}(\mathcal{C}_\Omega{,dxdy})$  such that
$w_{k} \rightharpoonup u_\lambda$ weakly in the space $\mathcal{X}^{s}_{\sum_{D}}(\mathcal{C}_\Omega{,dxdy})$, {up to a subsequence if necessary.
Put}  $W_k := (w_k - u_\lambda)$. {We claim} that   $w_k \to u_\lambda$ strongly in ${\mathcal{X}^{s}_{\sum_{D}}(\mathcal{C}_\Omega{,dxdy})}$
as $k\to \infty$. {Otherwise}, we {may} assume $$\|W_k\|^2_{\mathcal{X}^{s}_{\sum_{D}}(\mathcal{C}_\Omega{,dxdy})} \to a^2 \neq 0\;\text{and}\; \int_{\Omega} |W_k(x,0)|^{2^*_s}dx \to b^{2^*_s}\geq 0$$
as $k \to \infty$. By {the} Brezis-Lieb lemma we have
$$\lim_{k \to \infty}\Big(\|u_\lambda\|^2_{\mathcal{X}^{s}_{\sum_{D}}(\mathcal{C}_\Omega{,dxdy})} -\|w_k\|^2_{\mathcal{X}^{s}_{\sum_{D}}(\mathcal{C}_\Omega{,dxdy})}  + \|u_\lambda- w_k\|^2_{\mathcal{X}^{s}_{\sum_{D}}(\mathcal{C}_\Omega{,dxdy})} \Big) = 0.$$
{Since} $w_k\in\mathcal{N}_{\lambda}^{+}$, we obtain
\begin{equation}\label{eq8}
0  = \lim_{k \to \infty} \Phi'_{w_k}(1) = \Phi'_{u_\lambda}(1)+a^2 -b^{2^*_s},
\end{equation}
which implies
$$
\|u_\lambda\|^2_{\mathcal{X}^{s}_{\sum_{D}}(\mathcal{C}_\Omega{,dxdy})}+a^2 = \lambda \int_{\Omega}|u_\lambda(x,0)|^{1-q}dx + \int_{\Omega}|u_{\lambda}(x,0)|^{2^*_s}dx
+b^{2^*_s}.
$$
Before moving further, we claim that $u_\lambda \in \mathcal{X}^{s}_{\sum_{D}}(\mathcal{C}_\Omega{,dxdy})$ can not  be identically zero over ${\mathcal{C}_\Omega}$. If $u_\lambda\equiv 0$
 and $w_k\to u_\lambda$ strongly in ${\mathcal{X}^{s}_{\sum_{D}}(\mathcal{C}_\Omega{,dxdy})}$, then {Lemma} \ref{new-1} gives $0 > \inf J(\mathcal{N}^+_{\lambda}) = J(0)=0$,
giving a contradiction. On the other hand, since $a\neq 0$, so $u_\lambda \equiv 0$, \eqref{eq8} and $w_k\in \mathcal N_\lambda^+$ give
\begin{equation}\label{eq4.2new}
\inf_{w\in \mathcal{N}_{\lambda}^{+}} J(w) = \lim_{k\to \infty}J(w_k)
= \frac{a^2}{2} - \frac{b^{2^*_s}}{2^*_s}+ \frac{1}{1-q}\left( b^{2^*_s} -a^2 \right)
= \frac{a^2}{2} - \frac{b^{2^*_s}}{2^*_s} .
\end{equation}
But, {the fact that} $\|w_k(x,0)\|^2_{2^*_s}S(s,N) \leq \|w_k\|^2_{\mathcal{X}^{s}_{\sum_{D}}(\mathcal{C}_\Omega{,dxdy})}$
 {ensures that}
$a^2 \geq S(s,N)b^{2} $.
{Furthermore,} $a^2=b^{2^*_s}$ by \eqref{eq8}. {Hence,} \eqref{eq4.2new} implies
\[
0 > \inf_{w\in \mathcal{N}_{\lambda}^{+}} J(w)
= \Big(\frac{1}{2}-\frac{1}{2^*_s}\Big)a^2
\geq \frac{s}{N}S^{\frac{N}{2s}}(s,N),
\]
which is again a contradiction. Hence $u_\lambda \not\equiv 0$ in $\mathcal C_\Omega$  and the claim is proved.

Now applying Lemma \ref{ 3.2}, we state that there exists
{$t_*$, $t^*$, with} $0 < t_* < t^*$ such that
$\Phi'_{u_{\lambda}}(t_*)= \Phi'_{u_{\lambda}}(t^*) = 0$ and
$t_{*}u_{\lambda} \in \mathcal{N}^{+}_{\lambda}$, $t^*u_{\lambda} \in \mathcal{N}^{-}_{\lambda}$. Based on this, three possible cases arises as follows.\smallskip

\noindent\textbf{Case (i) ($t^* < 1$)} We consider the function
\[
g(t) = \Phi_{u_{\lambda}}(t)+ \frac{a^2t^2}{2} - \frac{b^{2^*_s} t^{2^*_s}}{2^*_s},
\]
for $t >0$. From equation \eqref{eq8}, we get $ g'(1) = \Phi'_{u_{\lambda}}(1)+a^2-b^{2^*_s} = 0$
and
\[
 g'(t^*) = \Phi'_{u_{\lambda}}(t^*)+t^*a^2-(t^*)^{2^*_s-1}b^{2^*_s}
= {t^*}(a^2 - (t^*)^{2^*_s-2}b^{2^*_s}) > {t^*}(a^2 - b^{2^*_s})
 > 0
\]
from which we can conclude that $g$ is
{an} increasing function on $[t^*,1]$. Hence,
\begin{align*}
\inf_{w\in \mathcal{N}_{\lambda}^{+}} J(w)
&= \lim_{k \to \infty}J(w_k) \geq \Phi_{u_{\lambda}}(1) + \frac{a^2}{2}- \frac{b^{2^*_s}}{2^*_s}
 = g(1) > g(t^*)\\
& =\Phi_{u_{\lambda}}(t^*) + \frac{a^2(t^*)^{2}}{2}
 - \frac{b^{2^*_s}(t^*)^{2^*_s}}{2^*_s}
\geq \Phi_{u_{\lambda}}(t^*) + \frac{(t^*)^{2}}{2} (a^2 - b^{2^*_s})\\
& > \Phi_{u_{\lambda}}(t^*) > \Phi_{u_{\lambda}}(t_*)
\geq \inf_{w\in \mathcal{N}_{\lambda}^{+}} J(w),
\end{align*}
which is a contradiction.
\smallskip

\noindent\textbf{Case (ii) ($t^* \geq 1$ and $\frac{a^2}{2}- \frac{b^{2^*_s}}{2^*_s} < 0$)}
{The inequalities}
$(a^2/2 - b^{2^*_s}/{2^*_s}) < 0$ and $S(s,N)b^2 \leq a^2$ {give}
$$a^2>\Big(\frac{2}{2^*_s}\Big)^{\frac{2}{2^*_s-2}}S^{\frac{2^*_s}{2^*_s-2}}(s,N).$$
{For each
$w_0 \in \mathcal{N}^{+}_{\lambda}$ we have $\Phi'_{w_0}(1)=0$, so that}
\begin{align*}
0  < \Phi''_{w_0}(1)
& = \|w_0\|^2_{\mathcal{X}^{s}_{\sum_{D}}(\mathcal{C}_\Omega{,dxdy})}+ q \lambda \int_{\Omega} |w_0(x,0)|^{1-q} dx - (2^*_s-1)
\int_{\Omega}|w_0(x,0)|^{2^*_s}dx \\
 & = (1+q) \|w_0\|^2_{\mathcal{X}^{s}_{\sum_{D}}(\mathcal{C}_\Omega{,dxdy})} -  (2^*_s-1+q) \int_{\Omega} |w_0(x,0)|^{2^*_s}dx,
\end{align*}
which implies
\[
 (1+q)\|w_0\|^2_{\mathcal{X}^{s}_{\sum_{D}}(\mathcal{C}_\Omega{,dxdy})}  > (2^*_s-1+q)\int_{\Omega}|w_0(x,0)|^{2^*_s}dx
=  (2^*_s-1+q)\|w_0(x,0)\|^{2^*_s}_{{2^*_s}}.
\]
The definition of $S(s,N)$ yields that
\[
 \|w_0\|^{2}_{\mathcal{X}^{s}_{\sum_{D}}(\mathcal{C}_\Omega{,dxdy})}  \leq  \Big(\frac{1+q}{2^*_s-1+q} \Big)^{\frac{2}{2^*_s-2}}
S^{\frac{2^*_s}{2^*_s-2}}(s,N).
\]
 Thus, using the fact that $0<\lambda <\Lambda$ we have
\[
\sup\{ \|w\|^2_{\mathcal{X}^{s}_{\sum_{D}}(\mathcal{C}_\Omega{,dxdy})}: w \in \mathcal{N}^+_{\lambda}\}
\leq \Big(\frac{2}{2^*_s}\Big)^{\frac{2}{2^*_s-2}} S^{\frac{2^*_s}{2^*_s-2}}(s,N)
< a^2 \leq \sup\{ \|w\|^2_{\mathcal{X}^{s}_{\sum_{D}}(\mathcal{C}_\Omega{,dxdy})}: w \in \mathcal{N}^+_{\lambda}\},
\]
which is again a contradiction.

Thus the only case which possibly holds is the following-\\
\noindent \textbf{Case (iii) ($t^* \geq 1$ and $\frac{a^2}{2}- \frac{b^{2^*_s}}{2^*_s} \geq 0$)} Here, by \eqref{eq8}  we have
$$
\inf_{w\in \mathcal{N}_{\lambda}^{+}} J(w)
= J(u_\lambda)+\frac{a^2}{2}-\frac{b^{2^*_s}}{2^*_s}
\geq J(u_\lambda) = \Phi_{u_\lambda}(1) \geq \Phi_{u_\lambda}(t_*)
\geq \inf J(\mathcal{N}^+_{\lambda}) .
$$
Clearly, this is possible when $t_* = 1$ and $(a^2/2 - b^{2^*_s}/{2^*_s}) = 0$
which yields $a=0$ by equation $\eqref{eq8}$ and $u_\lambda \in \mathcal{N}^+_{\lambda}$.
This shows that $w_k \to u_\lambda$ strongly in $\mathcal{X}^{s}_{\sum_{D}}(\mathcal{C}_\Omega{,dxdy})$ and also
$J(u_\lambda) = \inf_{w\in \mathcal{N}_{\lambda}^{+}} J(w)$.
This ends the proof.
\end{proof} \medskip
\begin{proposition}\label{prop4.2}
$u_\lambda$  is a positive weak solution of $(P_\lambda^*)$.
\end{proposition}
\begin{proof}
Fix $\zeta \in C^{\infty}_c(\mathcal{C}_{\Omega})$. By Lemma \ref{le05}, since $f \in L^{\infty}(\mathcal{C}_\Omega) \cap \mathcal{C}^\alpha(\overline{\mathcal{C}_\Omega})$ for some $\alpha \in (0,1)$,
we can find a constant $ M' >0$ such that $u_\lambda \geq M'>0$ on
the support of $\zeta$.
Then $u_{\lambda}+\epsilon \zeta\ge 0$ for  $\epsilon$
small enough. {The} same arguments as in the proof of
Lemma \ref{3.5} {shows that} $ J(u_\lambda+\epsilon \zeta) \geq J(u_\lambda)$
for sufficiently small $\epsilon >0$. Hence, we obtain
\begin{align*}
0 & \leq \lim _{\epsilon \to 0} \frac{J(u_\lambda+\epsilon\zeta) - J(u_\lambda)}{\epsilon}\\
&= \kappa_{s} \int_{\mathcal{C}_{\Omega}} y^{1-2s} \langle{\nabla u_\lambda, \nabla \zeta} \rangle \,dx\,dy
 - \lambda \int_\Omega  u^{-q}_{\lambda}(x,0){\zeta(x,0)}\,dx
 - \int_\Omega u^{2^*_s-1}_{\lambda}(x,0){\zeta(x,0)}dx  \,.
\end{align*}
Since $\zeta \in C^{\infty}_c(\mathcal{C}_{\Omega})$ is arbitrary, {ths
implies} that $u_\lambda$
is a positive weak solution of $(P_\lambda^*)$.
\end{proof} \medskip

\subsection{Existence of {a} second solution}
In this second subsection, we will show that the infimum of $J$ over $\mathcal{N}_{\lambda}^{-}$ is achieved. We need some technical lemmas to prove this. Let us recall the family of extremal functions defined in \eqref{talenti} as follows
$$ u_{\epsilon}(x)= \frac{\epsilon^{\frac{N-2s}{2}}} {(\epsilon^2 + |x|^2)^{\frac{N-2s}{2}}},~~ x\in \mathbb{R}^N,$$
with arbitrary $\epsilon>0.$ Let ${\psi_{0}} \in C^{\infty}(\mathbb{R_{+}}\cup\{0\})$ be a non-increasing cut-off function  satisfying
\begin{equation*}
\psi_{0}(z) =
\begin{cases}   1,&~\text{if} ~0\leq z \leq \frac12 \\
                0,&~\text{if}~ z \geq 1.
                        \end{cases}
\end{equation*}
Without loss of generality, we assume {that} $0 \in \Omega$ and we choose $\rho >0$ {so small  that} $B^{+}_{\rho}(0) \subseteq \mathcal{C}_{\Omega}$, where  $B^{+}_{\rho}(0) = \{z\in \mathcal{C}_{\Omega} : 0\leq |z| \leq \rho\}$.
Define the function
$$\psi_{\rho}(p_1,p_2) = \psi_{0}\left(\frac{r_{p_1 p_2}}{\rho}\right),$$
where $r_{p_1 p_2}= |(p_1,p_2)| = (|p_1|^2 + p_2^2)^{\frac12}, \text{for}~ p_1, p_2 \in \mathbb{R}^N$. {Put}
$$\tilde{w}_\epsilon(x,y) :=  \psi_{\rho}(x,y) w_\epsilon(x,y)
\quad\mbox{for all }x,\,y\in\mathbb R^N,$$
where
$w_\epsilon$ is given by $\mathcal{E}_{s}[u_{\epsilon}]= w_\epsilon.$  The following comprises the main feature of $w_\epsilon$
$$ S(s,N)= \frac{\int_{\mathbb{R}^{N+1}_{+}} y^{1-2s}|\nabla\mathcal{E}_{s}[u_{\epsilon}](x,y)|^2dxdy}{\Bigg({\int_{\mathbb{R}^N}}|\mathcal{E}_{s}[u_{\epsilon}](x,0)|^{2^*_s} dx \Bigg)^{2/2^*_s} }$$
and
$$ \int_{\mathbb{R}^N}|\mathcal{E}_{s}[u_{\epsilon}](x,0)|^{2^*_s}dx = S^{\frac{N}{2s}}(s,N).$$

\begin{lemma}\label{5.1}
For sufficiently small $\epsilon >0$ it results
 $$\sup \{J(u_\lambda + t\tilde{w}_\epsilon): t\geq 0\}
< J(u_\lambda) + \frac{s}{N}S^{\frac{N}{2s}}(s,N),$$
where $u_\lambda$ is defined in Proposition \ref{u-lambda}.
\end{lemma}
\begin{proof}
Fix $\epsilon>0$. {Since} $0\leq\psi_\rho\leq1$ we {have}
\begin{equation*}
\begin{split}
\kappa_{s} \int_{\mathcal{C}_{\Omega}} y^{1-2s}|{\nabla \tilde{w}_\epsilon(x,y)|^{2}}  \,dx\,dy&\leq \kappa_{s} \int_{\mathbb{R}^{N+1}_{+}} y^{1-2s}|{\nabla {w}_\epsilon(x,y)|^{2}} \,dx\,dy\\
&+\kappa_{s} \int_{\mathcal{C}_{\Omega}} y^{1-2s}|{\nabla \psi_{\rho}(x,y)|^{2}}{ |{w}_\epsilon(x,y)|^{2}}\,dx\,dy\\
&+ 2\kappa_s\int_{\mathcal{C}_{\Omega}} y^{1-2s} \langle{\nabla \psi_{\rho}(x,y)}{ {w}_\epsilon(x,y)},\nabla{w}_\epsilon(x,y)\psi_{\rho}(x,y)\rangle\,dx\,dy
\end{split}
\end{equation*}
Now, following the arguments of
the proof of Lemma 3.4 of \cite{MR3912863}, we obtain
$$\kappa_{s} \int_{\mathcal{C}_{\Omega}} y^{1-2s}|{\nabla \tilde{w}_\epsilon(x,y)|^{2}}  \,dx\,dy
\leq S^\frac{N}{2s}(s,N)+  o\left(\left(\frac{\epsilon}{\rho}\right)^{N-2s} \right).$$
Thus, we can choose constant $c_1 > 0$ such that
 \begin{equation}\label{asymp-1}
 \kappa_{s} \int_{\mathcal{C}_{\Omega}} y^{1-2s}|{\nabla \tilde{w}_\epsilon(x,y)|^{2}}|  \,dx\,dy
\leq S^\frac{N}{2s}(s,N)+ c_{1}\left(\frac{\epsilon}{\rho}\right)^{N-2s}
\end{equation}
Concerning the critical exponent term we see that
\begin{align*}
\int_{\Omega} |\tilde{w}_\epsilon(x,0)|^{2^{\ast}_{s}}dx
& =\int_{\Omega} |\psi_{\rho}(x,0){w}_\epsilon(x,0)|^{2^*_s}dx\\
& =\int_{B^+_\frac{\rho}{2}} \Big|\psi_{0}  \left(\frac{|x|}{\rho}\right)  u_{\epsilon}(x)\Big|^{2^*_s}dx
 +\int_{B^+_{\rho} \setminus B^+_\frac{\rho}{2}}\Big|\psi_{0}  \left(\frac{|x|}{\rho}\right)  u_{\epsilon}(x)\Big|^{2^*_s}dx \\
& = \int_{\mathbb{R}^N} |u_\epsilon|^{2^*_s}dx
 - \int_{\mathbb{R}^N \setminus B^+_\frac{\rho}{2}} |u_\epsilon|^{2^*_s}dx
 + \int_{B^+_{\rho} \setminus B^+_\frac{\rho}{2}} \Big|\psi_{0}  \left(\frac{|x|}{\rho}\right)  u_{\epsilon}(x)\Big|^{2^*_s}dx\\
& \geq \int_{\mathbb{R}^N} |u_\epsilon|^{2^*_s}dx
 - \int_{\mathbb{R}^N \setminus B^+_\frac{\rho}{2}} |u_\epsilon|^{2^*_s}dx=\int_{\mathbb{R}^N} |u_\epsilon|^{2^*_s}dx
 - \int_{|x|>\frac{\rho}{2}} \frac{\epsilon^N}{(\epsilon^2+|x|^2)^N}dx \\
 & = \int_{\mathbb{R}^N} |u_\epsilon|^{2^*_s}dx
 - \int_{\frac{\rho}{2\epsilon}}^{\infty} \frac{t^{N-1}}{(1+t^2)^N}dt\geq \int_{\mathbb{R}^N} |u_\epsilon|^{2^*_s}dx
 - c\int_{\frac{\rho}{2\epsilon}}^{\infty} \frac{1}{t^{N+1}}dt,
\end{align*}
for an appropriate constant $c>0$.
Therefore
\begin{equation}\label{asymp-2}
   \int_{\Omega} |\tilde{w}_\epsilon(x,0)|^{2^{\ast}_{s}}dx \geq  S^\frac{N}{2s}(s,N) - c_2 \left(\frac{\epsilon}{\rho} \right)^N
\end{equation}
for some {constant} $c_2 >0.$
 Let us fix {$\eta$, with}
 $1 < \eta < \frac{N}{N-2s}$. {Then,}
 \begin{align*}
\int_{\Omega} |\tilde{w}_\epsilon(x,0)|^\eta dx
&= \int_{B^+_\rho}\Big|\psi_{0}  \left(\frac{|x|}{\rho}\right)  u_{\epsilon}(x)\Big|^\eta dx \\
& \leq  \epsilon^{\frac{(N-2s)\eta}{2}} \int_{B^+_{\rho}}
  {(\epsilon^2 +|x|^2)^{\frac{-\eta(N-2s)}{2}}} dx \\
& = c \epsilon^{\frac{-\eta(N-2s)}{2}} \int_0^{\frac{\rho}{\epsilon}} \frac{\epsilon^{N}t^{N-1}}{(1+t^2)^{\frac{\eta(N-2s)}{2}}}dt\leq c \epsilon^{\frac{N-\eta(N-2s)}{2}} \int_0^{\frac{\rho}{\epsilon}} \frac{t^{N-1}}{t^{{\eta(N-2s)}}}dt\\
& = c \epsilon^{\frac{-\eta(N-2s)}{2}} \left(\frac{\rho}{\epsilon}\right)^{N-\eta(N-2s)},
\end{align*}
for an appropriate constant $c>0$ depending on $\rho$.
Therefore
\begin{equation}\label{asymp-3}
    \int_{\Omega} |\tilde{w}_\epsilon(x,0)|^\eta dx \leq c_3 \epsilon^{\frac{\eta(N-2s)}{2}},
\end{equation}
for some {constant} $c_3 >0$ depending on $\rho$.
Now, we choose $\rho $ so small such that $\rho< 2\epsilon$
and we find that
 \begin{align*}
 \int_{B^+_\frac{\rho}{2}} |\tilde{w}_\epsilon(x,0)|^{2^*_s-1}dx
&= \int_{B^+_\frac{\rho}{2}} \Big|\psi_{0}\left(\frac{|x|}{\rho}\right)  u_{\epsilon}(x)\Big|^{2^*_s-1}dx  = \epsilon^{\frac{-(N+2s)}{2}} \int_{B^+_\frac{\rho}{2}}
  \Bigg|\frac{1}{(\epsilon^2 +|x|^2)^{\frac{N-2s}{2}}}\Bigg|^{2^*_s-1} dx\\
& = c\epsilon^{N-\frac{(N+2s)}{2}} \int^{\frac{\rho}{2\epsilon}}_0
  \frac{\theta^{N-1}}{(1+\theta^2)^{\frac{N+2s}{2}}} d\theta \geq c \epsilon^{\frac{(N-2s)}{2}} \int^1_0\frac{\theta^{N-1}}{(1+\theta^2)^{\frac{N+2s}{2}}} d\theta .
\end{align*}
Therefore
\begin{equation}\label{asymp-4}
    \int_{B^+_\frac{\rho}{2}} |\tilde{w}_\epsilon(x,0)|^{2^*_s-1}dx \geq c_4\epsilon^{\frac{(N-2s)}{2}},
\end{equation}
for some {constant} $ c_4 >0$.
We can find some constants
$\eta_1, \eta_2 >0$ and $m>0$, $M>0$ such that $m\leq u_\lambda$ a.e. so that  the following inequalities hold
\begin{gather*}
 \lambda \Big( \frac{(a+b)^{1-q}}{1-q} - \frac{a^{1-q}}{1-q} - \frac{b}{a^q} \Big)
\geq -\frac{\eta_1 b^{\eta}}{c_3},\quad\text{for all }a \geq m, b\geq 0,
\\
\Big(\frac{(a+b)^{2^*_s}}{2^*_s} - \frac{a^{2^*_s}}{2^*_s}-a^{2^*_s-1}b\Big)
\geq \frac{b^{2^*_s}}{2^*_s},\quad\text{for all }a,b \geq 0,
\\
 \frac{(a+b)^{2^*_s}}{2^*_s}-\frac{a^{2^*_s}}{2^*_s} - {a^{2^*_s-1}}b
\geq \frac{b^{2^*_s}}{2^*_s} + \frac{\eta_2ab^{2^*_s-1}}{c_4m(2^*_s-1)},
\quad\text{for all } 0\leq a \leq M,\; b\geq 1.
\end{gather*}
  We know $u_{\lambda}$ is a positive weak solution of
$(P_\lambda^*)$ and $0<q<1$, so using
the above inequalities and \eqref{asymp-1}-\eqref{asymp-4}, we obtain
\begin{align*}
& J(u_\lambda +t\tilde{w_\epsilon})-J(u_\lambda)\\
& = J(u_\lambda +t\tilde{w_\epsilon})-J(u_\lambda)
- t\Big(\kappa_{s} \int_{\mathcal{C}_{\Omega}} y^{1-2s} \langle{\nabla u_{\lambda}, \nabla \tilde{w_\epsilon}} \rangle \,dx\,dy
 - \lambda \int_\Omega  u_{\lambda}^{-q}(x,0)\tilde{w_{\epsilon}}(x,0)   \,dx \\
 &- \int_\Omega u_{\lambda}^{2^*_s-1}(x,0)\tilde{w_{\epsilon}}(x,0)\,dx \Big)
\\
 & = \frac{t^2 \kappa_{s}}{2} \int_{\mathcal{C}_{\Omega}} y^{1-2s} |\nabla \tilde{w_{\epsilon}}(x,y)|^2 \,dx\,dy
- \frac{1}{2^*_s}\Bigg[ \int_{\Omega} |(u_\lambda+t{\tilde{w_{\epsilon}})(x,0)}|^{2^*_s}-|u_\lambda(x,0)|^{2^*_s}\,dx\Bigg] \\
&\quad +t\int_{\Omega}u_{\lambda}^{2^*_s-1}(x,0)\tilde{w_{\epsilon}}(x,0)\,dx   - \lambda \Bigg[ \int_{\Omega}\frac{|(u_\lambda+t\tilde{w_{\epsilon}})(x,0)|^{1-q}
 -|u_\lambda(x,0)|^{1-q})}{1-q}\,dx\\
 &\quad \quad- t \int_{\Omega}u_{\lambda}^{-q}(x,0)\tilde{w_{\epsilon}}(x,0) \,dx \Bigg]\\
& \leq \frac{t^2}{2}\Bigg(S^{\frac{N}{2s}}(s,N)+c_1 \Big(\frac{\epsilon}{\rho}\Big)^{N-2s}\Bigg)
 -\frac{t^{2^*_s}}{2^*_s}\int_{\Omega}|\tilde{w_{\epsilon}}(x,0)|^{2^*_s}dx
 + \frac{\eta_1t^{\eta}}{c_3}\int_{\Omega}|\tilde{w_{\epsilon}}(x,0)|^{\eta}dx\\
& \leq \frac{t^2}{2}\Bigg(S^{\frac{N}{2s}}(s,N)+c_1 \Big(\frac{\epsilon}{\rho}\Big)^{N-2s}\Bigg)
 -\frac{ t^{2^*_s}}{2^*_s}\Bigg(S^{\frac{N}{2s}}(s,N)- c_2 \Big(\frac{\epsilon}{\rho}\Big)^{N}\Bigg)
 +\eta_1t^{\eta} \epsilon^{\frac{\eta (N-2s)}{2}}
 \end{align*}
for each $0\leq t\leq 1/2$. Since we can assume $t \tilde{w_{\epsilon}} \geq 1$,
for each $t \geq 1/2$ and $|x| \leq \frac{\rho}{2}$, we have
 \begin{align*}
& J(w_\lambda+t\tilde{w_{\epsilon}})-J(w_\lambda)\\
 & \leq \frac{t^2}{2}\Bigg(S^{\frac{N}{2s}}(s,N)+c_1 \Big(\frac{\epsilon}{\rho}\Big)^{N-2s}\Bigg)
-\frac{t^{2^*_s}}{2^*_s}\int_{\Omega}|\tilde{w_{\epsilon}}(x,0)|^{2^*_s}dx
-\frac{\eta_2t^{2^*_s-1}}{c_4 (2^*_s-1)}\int_{|x|\leq \frac{\rho}{2}}
 |\tilde{w_{\epsilon}}(x,0)|^{2^*_s-1} dx \\
  &\quad + \frac{\eta_1t^{\eta}}{c_3}\int_{\Omega}|\tilde{w_{\epsilon}}(x,0)|^{\eta}dx\\
 & \leq \frac{t^2}{2}\Bigg(S^{\frac{N}{2s}}(s,N)+c_1 \Big(\frac{\epsilon}{\rho}\Big)^{N-2s}\Bigg)
-\frac{t^{2^*_s}}{2^*_s-1}\Big(S^{\frac{N}{2s}}(s,N)-c_2 \Big(\frac{\epsilon}{\rho}\Big)^{N}\Big)
-\frac{\eta_2 t^{2^*_s-1}}{(2^*_s-1)} \epsilon^{\frac{(N-2s)}{2}}\\
 &\quad +\eta_1t^{\eta}\epsilon^{\frac{(N-2s)\eta}{2}}.
 \end{align*}
Now, we define a function $H_\epsilon : [0, \infty) \to \mathbb{R}$ by
$$
H_\epsilon(t)=\begin{cases}
\frac{t^2}{2}\Bigg(S^\frac{N}{2s}(s,N)+c_1\Big(\frac{\epsilon}{\rho}\Big)^{N-2s}\Bigg)
 -\frac{t^{2^*_s}}{2^*_s}\Bigg(S^{\frac{N}{2s}}(s,N)-c_2 \Big(\frac{\epsilon}{\rho}\Big)^{N}\Bigg)
 +\eta_1t^{\eta}
\epsilon^{\frac{(N-2s)\eta}{2}},\\
~~\text{for}~t\in[0,1/2)\\
\frac{t^2}{2}\Bigg(S^{\frac{N}{2s}}(s,N)+c_1 \Big(\frac{\epsilon}{\rho}\Big)^{N-2s}\Bigg)
-\frac{t^{2^*_s}}{2^*_s}\Bigg(S^{\frac{N}{2s}}(s,N)-c_2 \Big(\frac{\epsilon}{\rho}\Big)^{N}\Bigg)
-\frac{\eta_2 t^{2^*_s-1}}{(2^*_s-1)} \epsilon^{\frac{(N-2s)}{2}}\\
+\eta_1t^{\eta}\epsilon^{\frac{(N-2s)\eta}{2}}
~~,\;\text{for} ~t\in [1/2,\infty)
\end{cases}
$$
By some approximating computations, we can get that $H_\epsilon$ attains its maximum at
\[
t_{\epsilon} = 1-\frac{\eta_2\epsilon^{(N-2s)/2}}{(2^*_s-2)S^\frac{N}{2s}(s,N)}
+o \Bigg({\Big(\frac{\epsilon}{\rho}}\Big)^{\frac{(N-2s)}{2}}\Bigg)
\]
which states that
\begin{align*}
\sup\{J(u_\lambda+t\tilde{w}_\epsilon)-J(u_\lambda): t\geq 0 \}
&\leq \frac{s}{N}S^{\frac{N}{2s}}(s,N)-\frac{\eta_2\epsilon^{(N-2s)/2}}{(2^*_s-1)}
+o \Bigg({\Big(\frac{\epsilon}{\rho}}\Big)^{(N-2s)/2}\Bigg)\\
&< {\frac{s}{N}S^{\frac{N}{2s}}(s,N)}.
\end{align*}
This completes the proof.

\end{proof} \medskip
\begin{lemma} \label{lem5.2}
It holds that
 $\inf J(\mathcal{N}^-_{\lambda}) < J(u_\lambda)+\frac{s}{N} S^{\frac{N}{2s}}(s,N)$.
\end{lemma}
\begin{proof}
 Fixing sufficiently small $\epsilon >0$ as in {Lemma} $\ref{5.1}$, we
define functions $\alpha_1, \alpha_2: [0,\infty) \to \mathbb{R}$ by
\begin{gather*}
\begin{aligned}
\alpha_1(t)
&= \kappa_{s} \int_{\mathcal{C}_{\Omega}} y^{1-2s} |{\nabla u_{\lambda}+ t\tilde{w_\epsilon}}|^2 \,dx\,dy
 - \lambda  \int_\Omega |({u_{\lambda}+\tilde{w_{\epsilon}}})(x,0)|^{1-q}   \,dx
 -  \int_\Omega |({u_{\lambda}+\tilde{w_{\epsilon}}})(x,0)|^{2^*_s}\,dx ,
\end{aligned} \\
\begin{aligned}
\alpha_2(t)&= \kappa_{s} \int_{\mathcal{C}_{\Omega}} y^{1-2s} |{\nabla u_{\lambda}+ t\tilde{w_\epsilon}}|^2 \,dx\,dy
+\lambda q \int_\Omega |({u_{\lambda}+\tilde{w_{\epsilon}}})(x,0)|^{1-q}   \,dx\\
&\quad
-(2^*_s-1)\int_\Omega |({u_{\lambda}+\tilde{w_{\epsilon}}})(x,0)|^{2^*_s}\,dx .
\end{aligned}
\end{gather*} Denote $t_0 = \sup\{t \geq 0: \alpha_2(t)\geq 0 \}$. Then,
 $\alpha_2(0) = \Phi''_{u_\lambda}(1) >0$ and
$\alpha_2(t) \to -\infty$ as $t \to \infty$. which
imply that $0<t_0<\infty$.
As $\lambda \in (0, \lambda_*)$, we obtain $\alpha_1(t_0)>0$ and we can check that
$\alpha_1(t) \to -\infty$ as $t \to \infty$, so there exists
$t' \in (t_0, \infty)$ such that $\alpha_1(t')=0$.
It gives $\Phi''_{u_\lambda+t'\tilde{w_{\epsilon}}}(1)<0$, because $t'>t_0$
implies $\alpha_2(t')<0$. Hence, $ (u_\lambda+t'\tilde{w_{\epsilon}})\in \mathcal{N}^-_{\lambda}$
and Lemma $\ref{5.1}$   gives that
$$\inf J(\mathcal{N}^-_{\lambda}) \leq J(u_{\lambda}+t'\tilde{w_{\epsilon}})< J(u_\lambda)+\frac{s}{N} S^{\frac{N}{2s}}(s,N),$$
as required.
\end{proof} \medskip
\begin{proposition}
There exists $v_\lambda \in \mathcal{N}_{\lambda}^{-}$ satisfying
$J(v_\lambda) = \inf_{w\in \mathcal{N}_{\lambda}^{-}} J(w)$.
\end{proposition}

\begin{proof}
Let $(v_k)_k$ be a sequence in $\mathcal{N}^-_{\lambda}$ such that
$J(v_k) \to \inf J(\mathcal{N}^-_{\lambda})$ as $k \to \infty$. Then using
Lemma \ref{lem3.4}, we know that there {exists} $v_\lambda \in \mathcal{X}^{s}_{\sum_{D}}(\mathcal{C}_\Omega{,dxdy})$ such that
$v_k \rightharpoonup v_\lambda$ weakly in $\mathcal{X}^{s}_{\sum_{D}}(\mathcal{C}_\Omega{,dxdy})$ as $k\to \infty$
up to a subsequence if necessary. Set $z_k := v_k - v_\lambda$.
We claim that $v_k \to v_\lambda$ strongly in $\mathcal{X}^{s}_{\sum_{D}}(\mathcal{C}_\Omega{,dxdy})$. Assume that $\|z_k\|^2_{\mathcal{X}^{s}_{\sum_{D}}(\mathcal{C}_\Omega{,dxdy})} \to a^2$
and $\int_{\Omega}|z_k(x,0)|^{2^*_s}dx \to b^{2^*_s}$ as $k \to \infty$.
Using the Brezis-Lieb lemma, we obtain
\begin{equation} \label{eq5.1}
    \|v_\lambda\|^2_{\mathcal{X}^{s}_{\sum_{D}}(\mathcal{C}_\Omega{,dxdy})} +a^2 = \lambda\int_{\Omega}|v_\lambda(x,0)|^{1-q}dx
+ \int_{\Omega}|v_\lambda (x,0)|^{2^*_s}+b^{2^*_s}.
\end{equation}
We now claim that $v_\lambda $ is not identically zero in $\mathcal{C}_{\Omega}$. If $v_\lambda\equiv 0$ in $\mathcal{C}_\Omega$
and $a \neq 0$  then by \eqref{eq5.1} and definition of $S(s,N)$, we see that
\[
\inf J(\mathcal{N}^-_{\lambda}) = \lim_{k\to \infty} J(v_k) =
\frac{a^2}{2}-\frac{b^{2^*_s}}{2^*_s} \geq \frac{s}{N}S^\frac{N}{2s}(s,N)>0.
\]
Hence, Lemma $\ref{lem5.2}$ gives
$$ \frac{s}{N}S^\frac{N}{2s}(s,N) \leq \inf J(\mathcal{N}^-_{\lambda}) < J(v_\lambda)
+ \frac{s}{N}S^\frac{N}{2s}(s,N),$$
that is $J(v_\lambda) >0$.
This is a contradiction since $v_\lambda \equiv 0$.
Hence, the claim is proved.

Applying Lemma\ref{ 3.2}, we state that there exist
$t_*,\, t^*$, with
$0 < t_* < t^*$, such that
$\Phi'_{v_{\lambda}}(t_*)= \Phi'_{v_{\lambda}}(t^*) = 0$ and
$t_{*}v_{\lambda} \in \mathcal{N}^{+}_{\lambda}$, $t^*v_{\lambda} \in \mathcal{N}^{-}_{\lambda}$.
We consider the functions
$X,Y :(0,\infty) \to \mathbb{R}~~ \text{ defined by}$
\begin{equation}\label{eq11}
 X(t) = \frac{a^2t^2}{2}-\frac{b^{2^*_s}t^{2^*_s}}{2^*_s} \quad \text{and} \quad
 Y(t) = \Phi_{v_\lambda}(t)+ X(t).
\end{equation}
Then,  three possible cases {arise} as follows.\\
\textbf{Case (i)} $(t^* <1)$  We have $Y'(1)= \Phi'_{v_\lambda}(1)+ X'(1) = 0$, using
\eqref{eq11} and $Y'(t^*)= \Phi'_{v_\lambda}(t^*)+ X'(t^*)
= t^*(a^2-b^{2^*_s}{t^*}^{2^*_s-2}) \geq t^*(a^2-b^{2^*_s})>0$.
We can observe that $Y$ is increasing on $[t^*,1]$ and we have
\[
\inf J(\mathcal{N}^-_{\lambda}) = Y(1)> Y(t^*)
\geq J(t^* v_\lambda)+\frac{{t^*}}{2}(a^2-b^{2^*_s}) > J(t^* v_\lambda)
\geq \inf J(\mathcal{N}^-_{\lambda}),
\]
which is a contradiction.\\
\textbf{Case (ii)} $(t^* \geq 1$ \textbf{and} $b >0)$ Let ${t'}=(a^2/b^{2^*_s})^{\frac{1}{2^*_s-2}}$.
{It is easy to see} that $X$ attains its maximum at $t'$ and
$$
X(t')= \frac{a^2{t'}^2}{2}-\frac{b^{2^*_s}
{{t'}}^{2^*_s}}{2^*_s}
=\Big( {\frac{a^2}{b^{2^*_s}}}\Big)^{\frac{2^*_s}{2^*_s-2}}
\Big(\frac{1}{2}- \frac{1}{2^*_s} \Big) \geq S^{\frac{2^*_s}{2^*_s-2}}(s,N)\frac{s}{N}
= \frac{s}{N}S^\frac{N}{2s}(s,N).
$$
Also, $X'({t})= t(a^2-b^{2^*_s}t^{2^*_s-2}) >0 $ if $0<t<t'$
and $X'(t) <0$ if $t >{t'}$. Moreover, we know
$Y(1) = \max _{t>0} Y(t) \geq Y({t'})$ using the assumption
$\lambda \in (0, \lambda_*)$. If $t'\leq 1$, then we have
\[
\inf J(\mathcal{N}^-_{\lambda}) = Y(1)\geq Y(t')=J(t'v_\lambda)
+ X({t'}) \geq J(t_* v_\lambda)+\frac{s}{N}S^\frac{N}{2s}(s,N),
\]
which contradicts Lemma \ref{lem5.2}. Thus, we must have ${t'}>1$.
Since $Y'(t) \leq 0$ for $t\geq 1$, there holds
$\Phi''_{v_\lambda}(t) \leq -X'(t) \leq 0 $ for $1\leq t \leq {t'}$.
Then we have ${t'}\leq t_*$ or $t^*=1$.
If ${t'}\leq t_*$ then
\[
\inf J(\mathcal{N}^-_{\lambda})= Y(1) \geq  Y({t'})
=J({t'}v_\lambda)+ X({t'})
\geq J(t_*v_\lambda)+\frac{s}{N}S^\frac{N}{2s}(s,N),
\]
which {contradicts Lemma} \ref{lem5.2}. If $t^*=1$, then {$a^2=b^{2^*_s}$ by \eqref{eq5.1} and so}
\[
\inf J(\mathcal{N}^-_{\lambda})=Y(1)= J(v_\lambda)
+ \Big(\frac{a^2}{2}-\frac{b^{2^*_s}}{2^*_s}\Big)\geq J(v_\lambda)
+ \frac{s}{N}S^\frac{N}{2s}(s,N).
\]
This is again a contradiction.\\
\textbf{Case (iii)} ($t^* \geq 1$ \textbf{and} $b = 0$)  For this case,  we are left with the choice that
if $a \neq 0$, then $\Phi'_{v_\lambda}(1) = -a^2 <0$ and
$\Phi''_{v_\lambda}(1) = -a^2<0$ due to \eqref{eq5.1} which contradicts $t^* \geq 1$.
Thus, $a=0$ {and} $v_k \to v_\lambda$ strongly in ${\mathcal{X}^{s}_{\sum_{D}}(\mathcal{C}_{\Omega},dxdy)}$.
Consequently, $v_\lambda \in \mathcal{N}^-_{\lambda}$ and
$\inf J(\mathcal{N}^-_{\lambda})= J(v_\lambda)$.
\end{proof} \medskip
\begin{proposition}\label{prop5.4}
 $v_\lambda$ is a positive weak solution of $(P_\lambda^*)$.
\end{proposition}

\begin{proof}
Let $\zeta \in C^{\infty}_c(\mathcal{C}_{\Omega})$. By
using Lemma \ref{le05},
since $f \in L^{\infty}(\mathcal{C}_\Omega) \cap \mathcal{C}^\alpha(\overline{\mathcal{C}_\Omega})$ for some $\alpha \in (0,1)$ we {there exists a} constant $k >0$ such that
$v_\lambda \geq {k}>0$ on $\operatorname{supp}(\zeta)$. Also, $t_{\epsilon} \to 1$
as $\epsilon \to 0+$, where $t_{\epsilon}$ is the unique positive
real number corresponding to $(v_\lambda+\epsilon \zeta)$ such that
$t_\epsilon (v_\lambda+\epsilon \zeta) \in \mathcal{N}^{-}_{\lambda}$.
Now, {Lemma \ref{3.5} gives}
\begin{align*}
0 & \leq \lim_{\epsilon \to 0}\frac{J(t_\epsilon(v_\lambda+\epsilon\zeta))
 - J( v_\lambda)}{\epsilon} \leq \lim _{\epsilon \to 0}
 \frac{J(t_{\epsilon}(v_\lambda+\epsilon\zeta)) - J(t_{\epsilon} v_\lambda)}{\epsilon}\\
&= \kappa_{s} {\int_{\mathcal{C}_{\Omega}}} y^{1-2s} \langle{\nabla v_\lambda, \nabla \zeta} \rangle \,dx\,dy
 - \lambda \int_\Omega  {v_{\lambda}}^{-q}(x,0) \zeta(x,0)\,dx
 - \int_\Omega {v_{\lambda}}^{2^*_s-1}(x,0) \zeta(x,0) dx.
  \end{align*}
Since $\zeta \in C^{\infty}_c(\mathcal{C}_{\Omega})$ is arbitrary, we conclude that
 $v_{\lambda}$ is positive weak solution of $(P_\lambda^*)$.
\end{proof} \medskip

\noindent
 \textbf{Proof of Therorem \ref{2}:} The assertion of Theorem \ref{2} is a direct consequence of Propositions~\ref{prop4.2}, \ref{prop5.4} and the fact that $u_\lambda(x,0)$ and $v_\lambda(x,0)$ {solve}  problem $(P_\lambda)$.

\section{Regularity of  weak solutions}

In this section, we shall discuss the regularity properties of weak solutions of $(P_\lambda^*)$. We start with the following lemma.

\begin{lemma} \label{lem6.1}
Let $w$ be a positive weak solution of $(P_\lambda^*)$
{and let $u \in \mathcal{X}^{s}_{\sum_{D}}(\mathcal{C}_\Omega{,dxdy})$. Then} $w^{-q}u \in L^{1}(\mathcal{C}_{\Omega})$ and 
\begin{gather*}
\kappa_{s} \int_{\mathcal{C}_{\Omega}} y^{1-2s} \langle{\nabla w(x,y), \nabla u (x,y)} \rangle \,dxdy
 - \int_\Omega \big(\lambda  w^{-q}(x,0)  + w^{2^*_s-1}(x,0)\big)u(x,0) \,dx = 0,
 \end{gather*}
 for each $u \in \mathcal{X}^{s}_{\sum_{D}}(\mathcal{C}_\Omega{,dxdy})$.
\end{lemma}
\begin{proof}
Let $w$ be a  weak solution of $(P_\lambda^*)$  and let
$u \in \mathcal{X}^{s}_{\sum_{D}}(\mathcal{C}_\Omega{,dxdy})$.
Consider first the case
when $u>0$ in  $\mathcal{C}_\Omega$. Then
by Lemma $\ref{lem2.6}$ there exists a sequence ${(u_m)_m \subset} \mathcal{X}^{s}_{\sum_{D}}(\mathcal{C}_\Omega{,dxdy})$ such that
$u_m \to u$ strongly in $\mathcal{X}^{s}_{\sum_{D}}(\mathcal{C}_\Omega{,dxdy})$, where each $u_m$ has compact support in
$\mathcal{C}_\Omega$ and $0 \leq u_1 \leq u_2 \leq \ldots$. Since each $u_m$ has compact support in $\mathcal{C}_{\Omega}$ and $w$ is a positive weak
solution of $(P_\lambda^*)$, we obtain
$$\kappa_{s} \int_{\mathcal{C}_{\Omega}} y^{1-2s} \langle{\nabla w(x,y), \nabla u_m (x,y)} \rangle \,dxdy =\int_\Omega \big(\lambda  w^{-q}(x,0)u_m(x,0) + w^{2^*_s-1}u_m(x,0)\big) \,dx,$$
for each~ $m\in \mathbb{N}.$
 Using the monotone convergence theorem, we obtain
 $$w^{-q}(x,0)u(x,0) \in L^{1}({\Omega})~\text{and}$$
$$\kappa_{s} \int_{\mathcal{C}_{\Omega}} y^{1-2s} \langle{\nabla w(x,y), \nabla u (x,y)} \rangle \,dxdy=\int_\Omega \big(\lambda  w^{-q}(x,0)u(x,0)\big) \,dx +\int_\Omega w^{2^*_s-1}(x,0)u(x,0)dx.$$

In the general case, $u = u^+ - u^-$ and $u^+, u^- \in \mathcal{X}^{s}_{\sum_{D}}(\mathcal{C}_\Omega{,dxdy})$.
The proof of the first part shows that the assertion of the lemma holds
for $u^+$ and $u^-$ and so for $u$. This concludes the proof.
\end{proof} \medskip

\begin{lemma}\label{lem6.2}
Let $w$ be a positive weak solution of $(P_\lambda^*)$. Then $w \in L^r(\mathcal{C}_\Omega)$ for each $r \in [1,\infty).$
\end{lemma}

\begin{proof}
Let $w$ is a positive weak solution {of $(P_\lambda^*)$.  We} claim that $w\in L^{2\beta}(\mathcal{C}_\Omega)$ implies that $w\in L^{2^*_{s}\beta}(\mathcal{C}_\Omega)$ for each $\beta \geq 1$. Suppose
that $w\in L^{2\beta}(\mathcal{C}_\Omega)$ for
some  $\beta \in[1,\infty)$. Fix $R\ge0$ and set $u= \min\{w^{\beta-1},R\}$.
Then, $wu, wu^2 \in \mathcal{X}^{s}_{\sum_{D}}(\mathcal{C}_\Omega{,dxdy})$. Using lemma \ref{lem6.1} and fixing a $k>0$, we obtain
\begin{equation*}
\begin{split}
&{\kappa_{s} \int_{\mathcal{C}_{\Omega}} y^{1-2s} {|\nabla  wu|^2}  \,dxdy
\leq \beta \kappa_{s} \int_{\mathcal{C}_{\Omega}} y^{1-2s} \langle{\nabla w(x,y), \nabla w(x,y)u^2 (x,y)} \rangle \,dxdy}\\
&= \beta\int_\Omega \big(\lambda  w^{-q}(x,0)+w^{2^*_s-1}(x,0)\big)w(x,0)u^2(x,0) \,dx\\
&\leq \lambda \beta \int_\Omega  w^{2\beta -q-1}(x,0)\,dx+ \beta \Bigg(\int_{w\leq k} w^{2\beta+2^*_s-2}(x,0)\,dx +\int_{w>k} w^{ 2^*_s-2}(x,0)w^2(x,0)u^2(x,0)\,dx\Bigg)\\
&\leq k_1 + \beta |\Omega|k^{2\beta+2^*_s-2} +k_2 \Bigg(\int_{w>k} w^{ (2^*_s-2)\frac{N}{2s}}(x,0)\,dx\Bigg)^{\frac{2s}{N}} \Bigg(\int_{w>k} (w^2(x,0)u^2(x,0))^{\frac{N}{N-2s}}\,dx\Bigg)^{\frac{N-2s}{N}}\\
&\leq k_1 + \beta |\Omega|k^{2\beta+2^*_s-2} +C \Bigg(\int_{w>k} w^{\frac{(2^*_s-2)N}{2s} }(x,0)\,dx\Bigg)^{2s/N} \Bigg (\kappa_{s} \int_{\mathcal{C}_{\Omega}} y^{1-2s} {|\nabla  wu|^2}  \,dxdy\Bigg),
\end{split}
\end{equation*}
where $k_1$, $k_2$ and C are positive constants independent of both $k$
 and $R$. We can choose {$k$} such that
$$C \Bigg(\int_{w>k} w^{\frac{(2^*_s-2)N}{2s} }(x,0)\,dx\Bigg)^{2s/N}\leq \frac{1}{2}.$$
Then we have,
$$\kappa_{s} \int_{w^{\beta-1 }\leq R} y^{1-2s} {|\nabla w^{\beta } |^2}  \,dxdy\\
\leq \kappa_{s} \int_{\mathcal{C}_{\Omega}} y^{1-2s} {|\nabla wu|^2}  \,dxdy\\
\leq 2(k_1 +\beta |\Omega|k^{2\beta +2^*_s-2}).$$
As $R \to \infty$ then we find
$$\kappa_{s} \int_{\mathcal{C}_{\Omega}} y^{1-2s} {|\nabla w^{\beta } |^2}  \,dxdy
\leq 2(k_1 +\beta |\Omega|k^{2\beta +2^*_s-2}).$$
This implies that $w^{\beta} \in \mathcal{X}^{s}_{\sum_{D}}(\mathcal{C}_\Omega{,dxdy})$ and therefore
the Sobolev embedding theorem implies that $w \in L^{2^*_s\beta}(\mathcal{C}_\Omega)$.
An inductive argument shows that  $w \in L^r(\mathcal{C}_\Omega)$ for each $1 \leq r < \infty$.
 \end{proof} \medskip

 \begin{lemma}\label{fin-1}
Let $w$ be a positive weak solution of $(P_\lambda^*)$. Then $w \in L^{\infty}(\mathcal{C}_\Omega)$.
\end{lemma}
\begin{proof}
Suppose that $w$ is a positive weak solution of $(P_\lambda^*)$.
Then for $\lambda>0$ we have\newline
$\Big(\lambda ((w-1)^+)^{-q} + ((w-1)^+)^{2^*_s-1} \leq \lambda +  ((w-1)^+)^{2^*_s-1}\Big)$.
Thus
\begin{equation*}
\begin{split}
&\kappa_{s} \int_{\mathcal{C}_{\Omega}} y^{1-2s} \langle{\nabla (w-1)^+, \varphi} \rangle \,dx dy
= \int_{w\geq 1}y^{1-2s} \langle{\nabla (w-1), \varphi} \rangle \,dx dy\\
&\leq \int_{ w\geq1}\big(\lambda+w^{2^*_s-1}(x,0)\big)\varphi(x,0)\,dx,
\leq \int_{\Omega}\big(\lambda+w^{2^*_s-1}(x,0)\big)\varphi(x,0)\,dx
\end{split}
\end{equation*}
for each $0\leq \varphi \in C^{\infty}_0(\mathcal{C}_\Omega)$. Applying
Lemma \ref{lem6.2} and following proof of Theorem 3.10 of \cite{MR3023003}, we assert that $w \in L^{\infty}(\mathcal{C}_{\Omega})$.
\end{proof} \\

\noindent
 \textbf{Proof of Therorem \ref{3}:} The assertion of Theorem \ref{3} is a direct consequence of Lemma~\ref{fin-1} and equivalence of problems \eqref{1} and \eqref{p}.

\section{Acknowledgement}
T. Mukherjee acknowledges the support of  Start up Research Grant from DST-SERB, sanction no. SRG/2022/000524.

P. Pucci is a  member of the {\em Gruppo Nazionale per l'Analisi Ma\-te\-ma\-ti\-ca, la Probabilit\`a e le loro Applicazioni}
(GNAMPA) of the {\em Istituto Nazionale di Alta Matematica} (INdAM) and was partly supported by the INdAM -- GNAMPA Project
{\em Equazioni alle derivate parziali: problemi e mo\-del\-li} (Prot\_U-UFMBAZ-2020-000761).

\end{document}